\documentclass[10pt]{amsart}

\usepackage{amsmath,amssymb}
\usepackage{mathrsfs}
\usepackage[active]{srcltx}
\usepackage{graphicx,color}
\usepackage{epic}
\usepackage{pstricks}
\usepackage{comment}
\usepackage{amsthm}

\newtheorem{theorem}{Theorem}[section]
\newtheorem{proposition}[theorem]{Proposition}

\newtheorem{lemma}[theorem]{Lemma}

\newtheorem{corollary}[theorem]{Corollary}
\theoremstyle{definition}

\newtheorem{remark}[theorem]{Remark}

\def\R{\mathbb{R}}

\def\N{\mathbb{N}}
\def\C{\mathbb{C}}
\def\ZZ{\mathbb{Z}}

\def\E{\mathbb{E}}

\def \sgn{\operatorname{sgn}}

\def \supp{\operatorname{supp}}

\def\cA{{\mathcal A}}
\def\cC{{\mathcal C}}
\def\cD{{\mathcal D}}

\def\cF{{\mathcal F}}

\def\cL{{\mathcal L}}
\def\cM{{\mathcal M}}

\def\cR{{\mathcal R}}
\def\cS{{\mathcal S}}

\def\cW{{\mathcal W}}

\def\cO{{\mathcal O}}

\def\cT{{\mathcal T}}
\def\cM{{\mathcal M}}
\def\cB{{\mathcal B}}
\def\cP{{\mathcal P}}

\def\tf{\widetilde{f}}

\def\m{\mathfrak{M}}
\def\T{{\mathbb{T}}}
\newcommand{\ud}{\mathrm{d}}

\newcommand{\X}{{X}}
\newcommand{\A}{{A}}
\newcommand{\B}{{B}}

\newcommand{\widecheck}{\check}

\newcommand{\bracket}[2]{\langle #1 , #2 \rangle}

\definecolor{darkred}{rgb}{0.7,0.1,0.1}

\title[The invariant subspaces of periodic Fourier multipliers]{The invariant subspaces of periodic Fourier multipliers with application to abstract evolution equations}

\author{Sebastian Kr\'ol \& Jarosław Sarnowski}

\address{Sebastian Kr{\'o}l, Faculty of Mathematics and Computer Science, Adam Mickiewicz University in Pozna{\'n}, ul. Uniwersytetu Pozna{\'n}skiego 4, 61-614 Pozna{\'n}, Poland}
\email{sebastian.krol@amu.edu.pl}

\address{Jarosław Sarnowski, Faculty of Mathematics and Computer Science, Nicolaus Copernicus University in Toru{\'n}, ul. Chopina 12/18, 87-100 Toru{\'n}, Poland}
\email{jsarnowski@doktorant.umk.pl}

\begin{document}




\begin{abstract}
By methods of harmonic analysis, we identify large classes of Banach spaces invariant of periodic Fourier multipliers with symbols satisfying the classical Marcinkiewicz type conditions. Such classes include general (vector-valued) Banach function spaces $\Phi$ and/or the scales of Besov and Triebel-Lizorkin spaces defined on the basis of $\Phi$.

We apply these results to the study of the well-posedness and maximal regularity property of an abstract second-order integro-differential equation, which  models various types of elliptic and parabolic problems arising in different areas of applied mathematics. 
In particular, under suitable conditions imposed on a convolutor $c$ and the geometry of an underlying Banach space $X$, we characterize the conditions on the operators $A$, $B$ and $P$ on $X$ such that the following periodic problem 
\[
\partial P \partial u + B \partial u + \A u + c \ast u = f \qquad \textrm{in } \cD'(\T; X)
\]
is well-posed with respect to large classes of function spaces. 
The obtained results extend the known theory on the maximal regularity of such problem.  

\end{abstract}

\renewcommand{\subjclassname}{\textup{2020} Mathematics Subject Classification}

\maketitle

\section{Introduction} 

Fourier multipliers with operator-valued symbols have found many applications in the theory of abstract evolution equations, in particular, in connection with solvability (well-posedness) and regularity of integro-differential equations. A large class of such equations can be modelled by the following abstract, degenerated second-order problem with a convolution term: 
\begin{equation}\label{AP}
(P u')' + B u' + A u + c \ast u = f. \tag{AP}\\
\end{equation}
Here, $A$, $B$, $P$ denote closed linear operators on a Banach space $X$ and 
$c$ is an operator-valued function. 

For particular forms of \eqref{AP}, the studies  of their well-posedness on diverse vector-valued function spaces have been increased with occurring two seminal papers by Amann \cite{Am97} and Weis \cite{We01}, where operator-valued counterparts of classical multiplier theorems for Besov and Lebesgue-Bochner spaces on $\R$ are provided. 
Those results indicated a right form of multiplier conditions (see \cite{ClPr01}, \cite{CdPSW00}), which have been further adapted to different situations; see, e.g. periodic multiplier results in \cite{ArBu02, ArBu04, StWe07}, which are relevant to this article.

In the literature one can extract two lines of research corresponding to such studies: namely, when \eqref{AP} is considered in the {\it euclidean} setting, that is, on $\R$ or $\R_+$, and in the {\it periodic} one, that is, on $\T:=\R/\ZZ$ (i.e., when the periodic conditions are imposed). 
Each of these lines is represented by a long series of papers; to mention a few representative results, see for the first one, e.g. \cite{Am97, We01, PrSi04, ArBaBu04, ChSr05, AuAx11a, ChFi14, ApKe20, KeAp20, Kr22}, and for the second one, e.g. \cite{ArBu02, ArBu04, KeLi04, Li06, BuFa08, LiPo08, Po09, BuFa09, KeLiPo09, HeLi12, LiPo13, FuLi14, BuCa17, BuCa18, BuCa19} (as well as the references therein).  
Such studies correspond to the well-known research program formulated by Amman in \cite[Section 3]{Am95} and labelled as {\it 'pairs of maximal regularity'}. 
In both settings, the basic idea for such studies is the same and relies on  multiplier theorems. 
Roughly, by the theory of vector-valued distributions, the well-posedness and regularity questions for \eqref{AP}, reduce to checking if corresponding Fourier multipliers with operator-valued symbols are bounded in a space under consideration. Such Fourier multiplier operators arise naturally via the representation formula for corresponding solution operators associated to a given form of \eqref{AP}.

In the euclidean setting, the so-called phenomenon of the {\it extrapolation of $L^p$-maximal regularity}, which can be simply considered as a special variant of the well-posedness with respect to various Banach function spaces, have been studied recently; see, e.g. \cite{PrSi04, AuAx11a, ChFi14, FaHyLi20, ChKr18, Kr22}.  
Beyond a natural theoretical interest in this phenomenon, the maximal regularity with respect to a more general function space is an important tool for the study of associated non-linear problems; see, e.g. \cite{KuWe04, PrSi07, AuAx11a}.  

In the periodic case, the well-posedness and maximal regularity were addressed mainly in the context of the classical Lebesgue, Besov, Triebel-Lizorkin spaces; see, e.g. the corresponding series of the references mentioned above.
The aim of this article is to extend such results to a much wider context of general Banach function spaces; see the main results of this paper, Theorems \ref{mr thm}, \ref{charact of mr}, \ref{charact of wp}, and \ref{charact with Z}.  In particular, we clarify the phenomenon of the extrapolation of $L^p$-maximal regularity for several periodic evolution equations modelled by \eqref{AP}; see Theorems \ref{charact of mr} and \ref{charact with Z}, as well as Section \ref{last}. These results, in particular, provide counterpart of the euclidean line of research mentioned above for the periodic situation. In addition, we provide a convenient framework for such studies, which reveals an underlying structure, allows to simplify and unify technicalities mainly resulting from the fact that we deal with higher order Marcinkiewicz's conditions, and allows to handle different questions (distributional or strong solvability, maximal regularity) in a unified manner. It is achieved with the help of two auxiliary results Theorem \ref{conditions} and Lemma \ref{lem to mr}. On the other-hand, it allows to extend many results from the related literature in several ways, by showing that assumptions usually made to get those results in the $L^p$-setting are sufficient for a large class of Banach function spaces; see Section \ref{last}.

To establish such results we make a revision of underlying multiplier results from \cite{ArBu02, ArBu04, BuKi04} applied in the context of $L^p$ setting. 
Roughly, our main periodic multiplier results, see Theorems \ref{new} and \ref{ext th},  assert that the standard multiplier conditions, which in the literature  usually  are imposed on the symbol of a Fourier multiplier to get its boundedness on the classical (vector-valued) Lebesgue, Besov or Triebel-Lizorkin spaces, are sufficient for its boundedness on much larger classes of spaces. Such classes include general Banach function spaces $\Phi$ and/or the scales of Besov and Triebel-Lizorkin spaces defined on the basis of $\Phi$. In particular, these results extend \cite[Theorem 1.3]{ArBu02},  \cite[Theorem 4.5]{ArBu04} and \cite[Theorem 3.2]{BuKi04}. 
Our proofs differ from the proofs presented in those papers. We rely on direct maximal function estimates; see the proofs of Lemma \ref{fact 2} and Theorem \ref{new}. 
 
We conclude with a remark on the strategy of the proof of our abstract extrapolation result, Theorem \ref{ext th}. Since the theory of periodic distributions presents a simplification in comparison to that on the real line, one could expect the same in the context of multiplier theorems. In fact, we show that such simplification is reflected mainly in the representation formulas for periodic Fourier multipliers; see Section \ref{per dist} for the further comments and Lemma \ref{aux obs}. In particular we do not address here  problems which appear in the euclidean setting; see, e.g. \cite[Problems 3.2 and 3.3]{HyWe07} and \cite[Section 4]{Kr22}. 
However, at some points (for instance, when the interplay between the regularity of the symbol and its Fourier transform is crucial), the euclidean setting presents some benefits in comparison to the periodic one. For this reason, instead of trying to prove some periodic results in a complete analogy to corresponding ones known in the euclidean setting, we deduce them from their euclidean counterparts via transference techniques; see the proof of Theorem \ref{ext th} (cf. also Lemma \ref{fact}). The tools for such transference methods are workout in Section \ref{transf sec}, which may be of independent interest; see  operator-valued variants of Jodeit's type theorem, Theorems \ref{Jodeit th} and \ref{variant of Jodeit's th}, as well as Lemmas \ref{weighted Lp} and \ref{RdeF}.

The organization of the paper is well-reflected by the titles of the following (sub)sections. 

\section{Auxiliary results}\label{spaces}

\subsection{Function spaces}
We refer the reader to the monograph by Bennett and Sharpley \cite{BeSh88} for the background on Banach function spaces. Here, we mention only several facts we use in the sequel.  

Let $\Psi$ be a Banach function space over $(G,\ud t)$, where $G$ denotes $\R$ or $\T$ equipped with the Lebesgue measure. It means that $\Psi$ is a Banach space, which is an order ideal of $L^0:=L^0(G, \ud t)$, i.e. for every $f\in L^0$ and $g\in \Psi$ if $|f|\leq |g|$, then $f\in \Psi$ and $\|f\|_\Psi \leq \|g\|_\Psi$. 
Here, $L^0$ stands for the space of all complex measurable functions on $G$ (as usual, any two functions equal almost everywhere are identified).
Moreover, $\Psi$ has Fatou's property, and by the Lorentz-Luxemburg theorem \cite[Theorem 2.7, p.10]{BeSh88}, $(\Psi')' = \Psi$ with equal norms. Here, $\Psi'$ denotes the {\it (K\"othe) dual} (or {\it associated space}) of $\Psi$; see \cite{BeSh88}.

We define the vector-valued variant of Banach function spaces $\Psi$ as follows. Let $X$ be a Banach space with norm $|\cdot|_X$. Set 
\[
\Psi (G; \X ) :=\{ f:G \rightarrow X \textrm{ strongly measurable}: \;\;  |f|_X\in \Psi\}
\]
and $\|f\|_{\Psi(G;X)} : = \||f|_X\|_{\Psi}$ for $f\in \Psi (X)$.
Throughout, the symbol $\Phi$ is reserved to denote a Banach function space over $(\T,\ud t)$. Note that if a function $e_0(\tau):= 1$ $(\tau \in \T)$ is in $\Phi$, then by the ideal property of $\Phi$ we get that $L^\infty(\T;X)\subset \Phi(\T; X)=:\Phi(X)$. In particular, if $\cP(\T;X)$ denotes the set of all $X$-valued polynomials on $\T$, i.e.
\[
\cP(X):=\cP(\T;X):= \left\{ \sum_{k=-N}^{N} e_k \otimes x_k: N\in \N, x_k\in X  \right\}
\]
then $\cP(X)\subset \Phi(X)$. Here, $(e_k \otimes x)(\tau):= e_k(\tau) x$, where $e_k (\tau):= \tau^k$ ($\tau\in \T$, $k\in \ZZ$, $x\in X$).

Moreover, we introduce a variant of vector-valued Besov and Triebel-Lizorkin spaces  corresponding to a Banach function space $\Phi$.

Let $\cD'(\T; X):= \cL(\cD,X)$, where $\cD:=\cD(\T)$ is a space of all complex-valued infinitely differentiable functions on $\T$ equipped in the usual locally convex topology. We refer to \cite[Section 3]{ScTr87} or \cite{Ed79} for the backgrounds on the scalar distributions on $\T$, and to \cite[Section 2]{ArBu04} for their vector-valued counterpart. 
For instance, relying on \cite[Proposition 2.1]{ArBu04}, it is readily seen that for each $\psi\in \cC(\R)$ with the compact support, the operator $\psi(\Delta)$ given by 
\[
\psi(\Delta)f := \sum_{k\in\ZZ} e_k\otimes \psi(k) \hat f(k) \qquad (f\in \cD'(\T; X))
\]
is in $\cL(\cD'(\T; X))$.

 Let $\{\psi_j\}_{j \in \N_0}$ be the resolution of the identity on $\R$ generated by a function $\psi\in \cC^\infty(\R)$ such that $\psi\equiv 1$ on $[-1,1]$ and $\supp \psi \subset [-2,2]$, i.e.  
\[
\psi_0 :=\psi, \qquad \psi_j := \psi(2^{-j} \cdot) - \psi(2^{-j+1} \cdot)\quad  \textrm{ for } j\in \N. 
\]
One can check that $\{\psi_j(\Delta)\}_{j\in \N_0}$ is the resolution of the identity operator on $\cD'(\T;X)$, i.e. for every $f\in \cD'(\T; X)$
\[
\sum_{j\leq N} \psi_j(\Delta)f = \psi(2^{-N} \Delta)f \rightarrow f \qquad \textrm{ in } \cD'(\T;X) \textrm{ as } N\rightarrow \infty.
\]

Let $\Phi$ be a Banach function space over $(\T, \ud t)$. For all $s\in \R$ and $q\in [1,\infty]$ we set (with usual modification when $q=\infty$):
\[
B^{s,q}_{\Phi}(\T, X):= \left\{f\in \cD'(\T; X):\,\,  \|f\|_{B^{s,q}_{\Phi}(\T, X)} :=\left( \sum_{j=0}^\infty  \|  2^{sj} \psi_j(\Delta) f \|^q_{\Phi(\T; X)} \right)^{1/q}<\infty\right\}, 
\]
\[ 
F^{s,q}_{\Phi}(\T,X) := \left\{ f\in \cD'(\T; X): \,\, \|f\|_{F^{s,q}_{\Phi}(\T,X)} :=\left\|  \left( \sum_{j=0}^\infty  | 2^{sj} \psi_j(\Delta)f |^q_X \right)^{1/q}  \right\|_\Phi < \infty \right\}.
\]

For $G=\R$ and a general Banach function spaces $\Psi$ over $(\R,\ud t)$, the corresponding generalized vector-valued Besov $B^{s,q}_\Psi(\R; X)$ and Triebel-Lizorkin spaces $F^{s,q}_\Psi(\R,X)$ were introduced in \cite[Section 2]{Kr22}. 
In the case when $G=\T$, the vector-valued counterpart of the classical Besov spaces $B^{s,q}_p(\T)$, i.e. for $\Phi=L^p$ over $(\T, \ud t)$, was introduced in \cite{ArBu04}. 
In the both cases ($G=\T$ or $G=\R$), under some additional assumption on $\Psi$, one can show that the spaces $B^{s,q}_\Psi(G;X)$ and $F^{s,q}_\Psi(G;X)$ share most of the properties of their well-known scalar prototypes which correspond to $\Psi = L^p$ and $X = \C$.  
For our further purposes, we only need a few basic properties of the spaces $B^{s,q}_\Phi(\T; X)$ and $F^{s,q}_\Phi(\T; X)$ corresponding to a general Banach function space $\Phi$ over $(\T,\ud t)$; see Proposition \ref{per Besov} below. 
For their proofs we need a preliminary result on the boundedness of Fourier multipliers $\psi_j(\Delta)$, $j\in \N_0$, (and other ones) on the underlying space $\Phi(X)$. Then, the proof of Proposition \ref{per Besov} can be carried out in an analogy to the non-periodic case when $G=\R$ as it has been treated in \cite{Kr22}; see \cite[Lemma 3.6]{Kr22} and \cite[Lemma 5.3]{Kr22}.
However, it should be pointed out that in a comparison to the proofs of some results in the case $G=\R$, the proofs of their periodic counterparts admit an essential simplification, which we indicate below; see also Subsection \ref{per dist}.

As it could be already noted above, we omit '$\T$' in the symbols of spaces over $(\T, \ud t)$. Similarly, in the scalar case, i.e. when $X=\C$, $\C$ is also omitted in the corresponding symbols. For instance, $B^{s,q}_\Phi(X)$ stands for $B^{s,q}_\Phi(\T;X)$ and $\cC^\infty_c(\R)$ denotes $\cC^\infty_c(\R;\C)$, etc.

\subsection{Preliminary results on boundedness of periodic multipliers} 
For a polynomially bounded sequence $m:\ZZ\rightarrow \cL(X,Y)$ we write $\check m$ to denote the corresponding periodic distribution in $\cD'(\cL(X,Y))$, i.e.
\[
\widecheck m := \sum_{k\in \ZZ} e_k\otimes  m(k).
\]

\begin{lemma}\label{fact} Let $m:\ZZ \rightarrow \cL(X,Y)$ be such that $\check m$ is in $L^{\infty}(\cL(X,Y))$. Then, for every Banach function space $\Phi$ over $(\T,\ud t)$ such that $L^\infty\subset \Phi \subset L^1$ the operator $m(\Delta)$ given by 
\[
m(\Delta)f := \sum_{k} e_k \otimes m(k) \hat f (k)  \quad f\in \cD'(X)
\] 
is in $\cL(\Phi(X),\Phi(Y))$ with $\|m(\Delta)\|_{\cL(\Phi(X),\Phi(Y))}\leq c_\Phi \|\widecheck{m}\|_{L^\infty}\|\chi_\T\|_\Phi$, where $c_\Phi$ denotes the norm of embedding operator from $\Phi$ into $L^1$. 

\end{lemma}

\begin{proof} A standard argument shows that $\Phi\hookrightarrow L^1$, i.e. $\|g\|_{L^1}\leq c_{\Phi} \|g\|_\Phi$ ($g\in \Phi$). Since for every $\eta \in L^\infty$, $g\in \Phi$ and $\tau \in \T$ we have
\[
| (\eta \ast g) (\tau)|\leq c_\Phi \|\eta\|_{L^\infty} \|g\|_\Phi
\]
we infer that 
\[
\|\eta \ast g \|_\Phi\leq c_\Phi \|\eta\|_{L^\infty} \|
\chi_\T\|_\Phi \|g\|_{\Phi},
\] where $\chi_\T$ is the characteristic function of $\T$.
Since for every $f\in L^1(X)$ we have
\[
 | (m(\Delta)f) (\tau) |_Y = |(\widecheck m \ast f) (\tau) |_Y \leq \big(\| \widecheck m \|_{\cL(X,Y)} \ast |f|_X\big) (\tau)\qquad (\tau \in \T),
\]
 the proof is complete.   
\end{proof}

In particular, Lemma \ref{fact} shows that each operator $\psi_j(\Delta)$, $j\in \N_0$, is bounded on $\Phi(X)$ if $L^\infty\subset \Phi \subset L^1$. To show their uniform boundedness on $\Phi(X)$ we need an additional assumption on the boundedness of the Hardy-Littlewood maximal operator on $\Phi$.

Recall that the Hardy-Littlewood maximal operators $M_\T$ and $M_\R$ are defined by 
\[
M_\T f(\tau) := \sup_{\epsilon>0} \frac{1}{\epsilon}\int_{\Gamma(\tau, \epsilon)}
 |f(\zeta) | \, |\ud \zeta| \qquad (\tau \in \T)
\] for $f\in L^1(\T)$, where $\Gamma(\tau, \epsilon):= \T\cap \{z\in \C: |z-\tau| \leq \epsilon \}$,    and 
\[
M_\R f(t) := \sup_{\epsilon > 0} \frac{1}{2\epsilon} \int_{[t-\epsilon, t+\epsilon]} |f(s)| \ud s \qquad (t\in \R)
\] for $f\in L^1_{loc} (\R)$. In the view of the standard identification between the function $f$ on $\T$ with its  $2\pi$-periodic extension $\tf$ on $\R$, $\tf(t):=f(e^{it})$, $t\in\R$,  there exists a constant $c>0$ such that 
\[
c^{-1} (M_\R\tf)(t)\leq (M_\T f)(e^{it})\leq  c (M_\R\tf)(t) \qquad (f\in L^1(\T), t\in\R).
\]

Note that the assumption that $M_\T$ is bounded on a Banach function space $\Phi$ implies that $L^\infty\subset \Phi \subset L^1$.  Indeed, if $f\in \Phi\setminus\{0\}$, then there exists a constant $c>0$ and a measurable subset $A$ of $\T$ such that $|f|\geq c\chi_A$, i.e. $\chi_A\in \Phi$. Hence, by the boundedness of $M_\T$ on $\Phi$, we get that $M_\T \chi_A \geq \frac{|A|}{2 \pi}$. Consequently, by the ideal property of $\Phi$, we obtain that $L^\infty\subset \Phi$.

The following lemma provides a periodic counterpart of \cite[Chapter 2,~(17)~p.57]{St93}.
For its proof we need a vector-valued variant of Fej\'er's theorem, which asserts that for an arbitrary Banach space $X$ and $g\in L^1(\T; X)$, if $S_l(g):= \sum_{|k|\leq l} e_k\otimes \hat g(k)$ for $l\in \N_0$, then 
\[
 \frac{1}{N+1} \sum_{l = 0}^{N} S_l(g)\rightarrow g \quad \textrm{ as } N\rightarrow \infty  \textrm{ in } L^1(\T;X).
\]
Its proof follows the lines of the proof of its scalar prototype almost verbatim.

\begin{lemma}\label{fact 2} 
\emph{(i)} Let $\eta \in \cC_c(\R; \cL(X,Y))$ be such that $\|\cF^{-1}\eta(t)\|_{\cL(X,Y)}\leq \phi(t)$, $t\in \R$, for an even, radially decreasing, integrable function $\phi$ on $\R$. Then, there exists a constant $c>0$ such that for every  $f\in L^1(\T; X)$ we have
\begin{equation}\label{estim for M}
|\eta(\Delta)f(\tau)|_{Y} \leq c \|\phi\|_{L^1(\R)}  (M_\T |f|_X)(\tau) \qquad (\tau  \in \T). 
\end{equation} 

\emph{(ii)} Let $\eta \in \cC_c^\infty(\R)$ and set $\eta_\epsilon := \eta(\epsilon \cdot)$, $\epsilon>0$. Then, for every Banach function space $\Phi$ over $(\T,\ud t)$ such that $M_\T$ is bounded on $\Phi$, the operators   
$\eta_\epsilon(\Delta)$, $\epsilon>0$, are uniformly in $\cL(\Phi(X))$.
\end{lemma}

\begin{proof} (i) Let $f\in L^1(\T;X)$. 
It is readily seen that for every $t\in \R$ the integral $\int_\R (\cF^{-1}\eta)(s)\widetilde f(t-s) \ud s =: (\cF^{-1}\eta \ast \widetilde f) (t)$ is absolutely convergent. Suppose that $\supp \eta \subset [-K, K]$ for some $K \in \ZZ$. Note that for 
every $t\in \R$ and $l>K$ we have 
\begin{align*}
\eta(\Delta)f(e^{it}) & = \sum_{|k|\leq l} e^{itk} \int_\R e^{-isk} \cF^{-1}\eta(s)  \ud s \hat f(k)\\
& = \int_\R \cF^{-1} \eta(s) \sum_{|k|\leq l} e^{i(t-s)k} \hat f(k)  \ud s. 
\end{align*}
Therefore, following the notation introduced in Fej\'er's theorem,  for $N>K$ and $t\in \R$ we get 
\begin{equation}\label{Fejer}
\frac{1}{N+1} \sum_{l=K}^{N}S_l(\eta(\Delta)f)(e^{it}) = \int_\R \cF^{-1} \eta(s) \frac{1}{N+1} \sum_{l=K}^{N} S_l(g_t)(e^{i s}) \ud s,
\end{equation}
where $g_t(\tau):= f(e^{it}\bar \tau)$, $\tau\in \T$. 
By our assumption on $\cF^{-1}\eta$ it is straightforward to show that the right-hand side of \eqref{Fejer} converges to $(\cF^{-1}\eta \ast \widetilde f)  (t)$ as $N\rightarrow \infty$. Since the left-hand side is equal $\frac{N-K+1}{N}\eta(\Delta)f(e^{it})$, we infer that for every $t\in \R$
\[
\left| \eta(\Delta)f(e^{it})\right|_Y = \left|(\cF^{-1}\eta \ast \widetilde f)  (t) \right|_Y \leq (\phi \ast |\widetilde f|_X)(t).
\]

Further, note that for each $\epsilon>0$ there exists a function $\phi_\epsilon:= \sum_{j=0}^N c_j \chi_{B_j}$, where $B_j$ denotes an interval with the center in $0$ and $c_j>0$, such that $0\leq \phi_\epsilon\leq \phi$, $\|\phi - \phi_\epsilon\|_{L^\infty} <\epsilon$ and $\|\phi - \phi_\epsilon\|_{L^1} <\epsilon$ . One can readily check that for every $\delta>0$ there exists $\epsilon$ such that for every $t\in \R$ we have
\[
|(\phi\ast |\widetilde f|_X) (t)| \leq ((\phi - \phi_\epsilon) \ast |\widetilde f|_X) (t) + (\phi_\epsilon \ast |\widetilde f|_X) (t) \leq \delta + \|\phi_\epsilon\|_{L^1} (M_\R |\widetilde f|_X) (t).
\]
Since there exists $c>0$, independent of $f\in L^1(\T;X)$, such that $(M_\R|\widetilde f|_X)(t)\leq c (M_\T |f|_X)(e^{it})$ for every $t\in \R$, we get \eqref{estim for M}. 

(ii) Note that $(\cF^{-1}\eta_\epsilon)(t) = \frac{1}{\epsilon}(\cF^{-1}\eta) (\frac{t}{\epsilon})$ for every $\epsilon>0$ and $t\in \R$. Moreover, there exists a constant $C>0$ such that $|(\cF^{-1}\eta) (t)|\leq \frac{C}{1+t^2}$, which yields $|(\cF^{-1}\eta_\epsilon) (t)|\leq \frac{\epsilon^{-1}  C}{1+(\epsilon^{-1}t)^2}=:\phi_\epsilon(t)$ for every $t\in \R$ and $\epsilon>0$. Since $\|\phi_\epsilon\|_{L^1} = C\pi$ for every $\epsilon$. Therefore, \eqref{estim for M} gives the desired claim.
\end{proof}

\subsection{Fundamental properties of generalized Besov and Triebel-Lizorkin spaces}

Here, we collect the fundamental properties of such spaces, which play a role in our further studies; see Proposition \ref{per Besov} below.

We start with some preliminaries. Note that for every $s\in \R$ the space $B^{-s,1}_\Phi(X^*)$ embeds into $(B^{s,\infty}_\Phi(X))^*$. Indeed, this embedding is given by the following duality pairing: for each $f \in B^{s,\infty}_{\Phi}(X)$ and $g\in B^{-s,1}_{\Phi'}(X^*)$ we set
\begin{equation}\label{duality}
\bracket{g}{f}:= \sum_{j,l\in \N_0} \int_{\T} \bracket{\psi_l(\Delta)g(t)}{\psi_j(\Delta)f(t)}_{X^*,X} \ud t.
\end{equation}
Note that 
\[
\bracket{g}{f}:= \sum_{r\in \{\pm 1, 0\}} \sum_{l\in \N_0} 
\bracket{\breve{\psi}_{j}(\Delta)\psi_l(\Delta)g}{\chi_j(\Delta)f}_{\Phi'(X^*),\Phi(X)},
\] where $\breve{\psi}_j:=\psi_j(-\cdot)$ and $\chi_j:=\psi_{j-1}+\psi_j+\psi_{j+1}$ $(j\in \N_0)$ with $\psi_{-1} \equiv 0$ if $j=0$.

\begin{proposition}\label{per Besov} Let $X$ be a Banach space and $\Phi$ be a Banach function space over $(\T, \ud t)$. Then, the following assertions hold. 
\begin{itemize}
\item [(i)]  If $L^\infty \subset \Phi\subset L^1$, then for every 
\[
\E \in \left\{ B^{s,q}_\Phi,\, F^{s,r}_\Phi \,: s\in \R, q\in [1,\infty], r\in (1,\infty) \right\} 
\]
\begin{equation}\label{embedding}
\cP(X) \subset \E( X) \hookrightarrow \cD'(X).
\end{equation}
In particular, $\E(X)$ is a Banach space.

\item [(ii)] If $M_\T$ is bounded on $\Phi$, then for all $s\in \R$,  $\cP(X)$ is a dense subset of $B^{s,q}_\Phi(X)$ in the norm topology  if $q\in[1,\infty)$, and in the $B^{-s,1}_{\Phi'}(X^*)$-topology of $B^{s,\infty}_{\Phi}(X)$ if $q=\infty$.

 More precisely, for every $f\in B^{s,q}_\Phi(X)$  we have that 
\[
 \sum_{0\leq j\leq N} \psi_j(\Delta) f = \psi(2^{-N} \Delta)f \rightarrow f  \quad \textrm{ as } N\rightarrow \infty.
\]  
in $B^{s,q}_\Phi(X)$ for each $q<\infty$, and in the other case, i.e. $q=\infty$, the convergence holds in the $\sigma(B^{s,\infty}_{\Phi}(X), B^{-s,1}_{\Phi'}(X^*))$-topology.
 
\item [(ii')] If $M_\T$ is bounded on $\Phi$ and its dual $\Phi'$, then $\cP(X)$ is a dense subset of $F^{s,q}_\Phi(X)$ and for ever $f\in F^{s,q}_\Phi(X)$  we have that 
\[
 \sum_{0\leq j\leq N} \psi_j(\Delta) f = \psi(2^{-N} \Delta)f \rightarrow f  \quad \textrm{ as } N\rightarrow \infty.
\]    

\item [(iii)] If $M_\T$ is bounded on $\Phi$, then for every distribution $f\in \cD'(X)$, $f$ belongs to $B^{s,q}_{\Phi}( X)$ if and only if $\partial f$ belongs to $B^{s-1,q}_\Phi(X)$. Moreover, the function 
\begin{equation}\label{equiv norm}
B^{s,q}_\Phi( X) \ni f \mapsto  \left\| \partial f \right \|_{B^{s-1,q}_\Phi(X)}
\end{equation}
is an equivalent norm on $B^{s,q}_\Phi( X)$.
\end{itemize}
\end{proposition}

\begin{remark} (a) Compared to the case of the real line $\R$ (see \cite[Lemma 3.6]{Kr22}), note that in the periodic case there is a {\it common} dense subset, i.e. $\cP(X)$, for all $\E(X)$ with 
\[
\E \in  \{ B^{s,q}_\Phi, \, F^{s,r}_\Phi\,: s\in \R, q\in [1,\infty], r\in (1,\infty) \}. 
\] 

(b) In contrast to \cite[Lemma 3.6]{Kr22}, in Proposition \ref{per Besov}, in the case of Besov spaces we do not assume that $M_\T$ is bounded on the dual of $\Phi'$. 
\end{remark}

By Lemma \ref{fact 2}, the proof of Proposition \ref{per Besov}, 
mimics the proof of the corresponding statements on $\R$, \cite[Lemmas 3.6 and 5.4]{Kr22}. For the convenience of the reader we provide some auxiliary observations which should be made.

\begin{proof}[The proof of Proposition \ref{per Besov}] 
 
(i) We start with the case when $\E = B^{s,q}_\Phi$. Since $\cP \subset L^\infty \subset \Phi$, the left inclusion in \eqref{embedding} holds readily. 
For the second one, note that for every $f\in B^{s,q}_\Phi(X)$ and $\phi\in \cD$ we have that 
\begin{align*}
\left|\left(\phi_j(\Delta)f \right)(\phi)\right|_X & =\left| \sum_{k\in \ZZ} \int_\T e_k \phi \ud t \, \psi_j(k)\hat f(k)\right|_X \\
& =  2\pi \left|\sum_{k\in \ZZ} \chi_j(k) \hat \phi(- k) \psi_j(k) \hat f(k)\right|_X \\
& \leq  2\pi \int_\T  \left|\sum_{k\in \ZZ} e_k\otimes \psi_j(k)\hat f(k) \right|_X \left|\sum_{l\in \ZZ} \chi_j(l) \hat \phi(-l) e_l \right| \ud t\\
&\leq 2\pi \left\| \psi_j(\Delta)\hat f \right\|_{\Phi(X)} \left\|\chi_j(\Delta) \breve\phi \right\|_{\Phi'}
\end{align*}
Here, $\breve{\phi}(e^{it})=\phi(e^{-it})$, $t\in [-\pi, \pi]$. Since $\lim_{N\rightarrow \infty} \sum_{j=0}^{N} \phi_j(\Delta) f = f$ in $\cD'( X)$,  we obtain that
 
\begin{align*}
| f(\phi) |_X & \leq \sum_{j\in \N_0} \left\| 2^{js} \psi_j(\Delta) f \right\|_{\Phi(X)} \| 2^{-js} \chi_j(\Delta)\breve\phi\|_{\Phi'} \\ 
& \leq \left( \sum_{j\in \N_0} \left\| 2^{jsq} \psi_j(\Delta) f \right\|^q_{\Phi(X)} \right)^{1/q} 
\left( \sum_{j\in \N_0} \| 2^{-js} \chi_j(\Delta)\breve\phi\|_{\Phi'}^{q'} \right)^{1/q'} \\ 
& \leq  \|f\|_{B^{s,q}_{\Phi}(X)} \left( \sum_{j\in \N_0} \| 2^{-js} \chi_j(\Delta)\check\phi\|_{\Phi'}^{q'} \right)^{1/q'}  
\end{align*} (with the usual modification when $q=\infty$). 
Take $\alpha > |s|+1$ and set $\rho(t):= (1+t^2)^{-\alpha}$, $t\in \R$. Then, for every $j\in \N_0$  and $\tau \in \T$ we have
\begin{align*}
2^{(|s|+1) j} \left| \cF^{-1}({\rho\chi_j}_{|\ZZ})(\tau)\right| & = \left| \sum_{k\in \ZZ} \tau^k \frac{2^{(|s|+1) j}}{(1+k^2)^\alpha}\chi_j(k) \right|\\
& \leq \sum_{2^{j-2} <k < 2^{j+2}} 4^\alpha 2^{-\alpha j} \leq 4^{\alpha + 1} 
\end{align*}
Therefore, by Lemma \ref{fact}(i), the operators $(2^{(|s|+1) j}\chi_j \rho)(\Delta)$, $j\in \N_0$, restrict to uniformly bounded operators in $\cL(\Phi')$.  

Since  $(\rho^{-1}(\Delta)\breve\phi)(\tau) = \sum_{k\in \ZZ} \tau^k (1+k^2)^\alpha \cF {\breve{\phi}}(k)$ $(\tau \in \T)$ is in $\cD \subset \Phi'$,  for every $j\in \N_0$ we infer that
\begin{align*}
\| 2^{j|s|} \chi_j(\Delta)\breve\phi \|_{\Phi'} & \leq 2^{-j} \| (2^{(|s|+1)j} \rho \chi_j)(\Delta)(\rho^{-1})(\Delta)\breve\phi  \|_{\Phi'}\\
& \leq 2^{-j} \| (2^{(|s|+1)j} \rho \chi_j)(\Delta)\|_{\cL(\Phi')} \|\rho^{-1} (\Delta)\breve\phi] \|_{\Phi'}\\
&\leq 2^{-j}4^{\alpha + 1} \|\rho^{-1} (\Delta)\breve\phi] \|_{\Phi'}.
\end{align*}
It yields $B^{s,q}_\Phi(X)\hookrightarrow \cD'(X)$. The proof in the case when $\E = F^{s,q}_\Phi$ follows closely the arguments presented above. Therefore, we omit it.\\

(ii) Note that for every $j\in \N$, $\psi_j(t)= \eta (2^{-j}t)$, where $\eta(t) :=\psi(t) - \psi(2t)$, $t\in \R$. Thus, by Lemma \ref{fact 2}(ii), the operators $\psi_j(\Delta)$, $j\in \N_0$, restrict to uniformly bounded operators in $\cL(\Phi(X))$. Now the proof mimics the lines of the proof of the corresponding statement in \cite[Lemma 3.6]{Kr22}. Therefore, we omit it.\\

(ii') Since $F^{s,q}_{\Psi}(X) = B^{s,q}_\Psi(X)$ for every $\Psi=L^q_w$ with $q\in (1,\infty)$ and $w\in A_q(\T)$, by the point $(ii)$, the statement $(ii')$ holds for such spaces. The proof for a general $\Phi$ extrapolates from these particular ones via an adaptation of the Rubio de Francia iteration algorithm.

Let $\cR$ and $\cR'$ denote the following sublinear, positive operators defined on $\Phi$ and $\Phi'$, respectively: 

\begin{equation}\label{algo ops}
\cR g := \sum_{k=0}^\infty \frac{M_\T^k |g|}{(2\|M_\T\|_{\Phi})^k} \quad (g\in \Phi) \quad \textrm{ and } \quad
\cR' h := \sum_{k=0}^\infty \frac{M_\T^k |h|}{(2\|M_\T\|_{\Phi'})^k} \quad (h\in \Phi')  
\end{equation}

Here, $M_\T^k$ stands for the $k$-th iteration of the Hardy-Littlewood operator $M_\T$, $M_\T^0 := I$, and $\|M_\T\|_\Phi := \sup_{\|g\|_\Phi \leq 1}\|M_\T g\|_{\Phi}$. 
Since 
\[
M(\cR g ) \leq 2 \|M\|_\Phi \cR g  \quad (g\in \Phi) \quad \textrm{and} \qquad M(\cR' h) \leq 2 \|M\|_{\Phi'} \cR h \quad (h\in \Phi'),
\]
for every $g\in \Phi$ and $h\in \Phi'$ the functions $\cR g$ and $\cR'h$ belong to Muckenhoupt's class $A_1(\T)$. Consequently, for every $q\in (1,\infty)$, $g\in \Phi$, $g\neq 0$, and $h\in \Phi'$, $h\neq 0$,  the function 
\begin{equation}\label{weight}
w_{g,h,q}:=(\cR g)^{1-q} \cR' h
\end{equation}
is in Muckenhoupt's class $A_q(\T)$ with the $A_q$-constant 
\begin{equation}\label{constants}
[w_{g,h,q}]_{A_q}\leq 2^q \|M\|_{\Phi}^{q-1} \|M\|_{\Phi'}.  
\end{equation}
Moreover, note that if $g\in \Phi(X)$, then $f\in L^q_w(X)$ for $w:= (\cR|f|_X)^{1-q} \cR'h$ with an arbitrary $h\in \Phi'$, $h\neq 0$. Indeed, since $|g|_X\leq \cR|g|_X$, we get 
\[
\int_\R |g|_X^q w \ud t \leq \int_\R \cR|g|_X \cR' h \ud t \leq \|\cR |g|_X\|_\Phi \| \cR'h\|_{\Phi'}<\infty.
\] 
By Lemma \ref{fact 2} combined with Muckenhoupt's theorem (see e.g. \cite{BoKa97})  we get that if
\[
\mu_{w,q}: = \sup_{j\in \N_0} \|\psi_j(\Delta)\|_{\cL(L^q_w(X))},
\]
then $\sup_{w\in \cW} \mu_{w,q} < \infty$ for each $\cW\subset A_q(\T)$ with $\sup_{w\in \cW}[w]_{A_q}<\infty$.

Now we are in a position to apply Rubio de Francia's extrapolation argument. Fix $q\in (1,\infty)$. Let $f\in F^{s,q}_\Phi(X)$ and $f_N:= \sum_{j\leq N} \psi_j(\Delta)f$ for all $N\in \N$. Note that $f_N \in \cP(X)$ and 
\[
\|f-f_N\|_{F^{s,q}_\Phi(X)} \lesssim  \sum_{r=-1}^1 \Big\|\big( \sum_{j\geq N} 2^{jsq}|\psi_{j+r}(\Delta)\psi_j(\Delta)f|_X^q\big)^{1/q}\Big\|_{\Phi}.
\]
For each $r\in \{ \pm 1, 0\}$ set
\[
 g_r:= g_{N,r,f} := \big( \sum_{j\geq N} 2^{jsq}|\psi_{j+r}(\Delta)\psi_j(\Delta)f|_X^q\big)^{1/q},  \quad \textrm{and}
\]
\[
 g : =  g_{N,f} := \big( \sum_{j\geq N} 2^{jsq}|\psi_j(\Delta)f|_X^q\big)^{1/q}\quad (N\in \N).
\]
Let $h\in \Phi'$ with $h\neq 0$ and $w:= w_{g, h, q}= \cR(g)^{1-q} \cR' h$. Then, 
\begin{align*}
\mu_{w,q} \|\cR\|_{\cL(\Phi)} \|g\|_\Phi \|\cR'\|_{\cL(\Phi')} \| h\|_{\Phi'} & \geq \mu_{w,q} \|\cR g\|_\Phi \|\cR'h\|_{\Phi'} \\
& \geq \mu_{w,q} \left( \int_\T \cR g \cR' h \, \ud t \right)^\frac{1}{q} \left(  \int_\T \cR g \cR' h \, \ud t \right)^\frac{1}{q'}\\
&\geq  \left( \int_\T g_r^q w \,\ud t \right)^\frac{1}{q} \left(  \int_\T \cR g \cR' h \ud t \right)^\frac{1}{q'}\\
&\geq \int_\T g_r  \cR' h \, \ud t\\
&\geq \int_\T g_r |h| \, \ud t.
\end{align*}
Since 
\[
\sup\left\{ [w_{g_{N,f},h,q}]_{A_q}: f\in F^{s,q}_\Phi(X), N\in \N, h\in \Phi'\setminus \{0\} \right\} < \infty,
\]
we get that 
\[
\mu: = \sup\left\{ \mu_{w,q}: \, w=w_{g_{N,f}, h, q} \textrm{ with } f\in F^{s,q}_\Phi(X), N\in \N, h\in \Phi'\setminus \{0\} \right\} < \infty.
\]
 Therefore, by the Lorentz-Luxemburg theorem, we infer that $g_r \in \Phi(X)$ ($r\in \{\pm 1 ,0\})$ and for all $N\in \N$ we have 
\[
 \| g_{N,r,f} \|_\Phi  \leq \mu  \|\cR\|_{\cL(\Phi)} \|\cR'\|_{\cL(\Phi')} \|f-f_N\|_\Phi. 
\]
Since, by Fatou's property of $\Phi$, $f_N - f \rightarrow 0$ in $\Phi(X)$ as $N\rightarrow \infty$, we get the desired claim. 

$(iii)$ First note that for every $f\in \cD'(X)$ and $j\in \N_0$, $\partial f = \rho(\Delta) f$, where  $\rho(t) = it$ $(t\in\R)$,  and  
\begin{align*}
 \|2^{(s-1)j} \psi_j(\Delta)\partial f\|_{\Phi(X)} & =  \|2^{(s-1)j} (\rho\chi_j)(\Delta)\psi_j(\Delta)f\|_{\Phi(X)}\\
& \leq   \|2^{-j} (\rho\chi_j)(\Delta)\|_{\cL(\Phi(X))} \|2^{sj}\psi_{j}(\Delta)f\|_{\Phi(X)}.
\end{align*}
Moreover, if we set $\eta(t):= t\psi(\frac{1}{2}t) - t\psi(4t)$ $(t\in \R)$, then for every $j\geq 2$ and $t\in \R$, we have
\[
\eta (2^{-j}t) = -i 2^{-j}(\rho \chi_j)(t).
\]
Thus, by Lemma \ref{fact 2}, we infer that $\sup_{j\in \N_0} \|2^{-j} (\rho\chi_j)(\Delta)\|_{\cL(\Phi(X))} <\infty$, and consequently $\partial f \in B^{s-1,q}_\Phi(X)$ if $f\in B^{s,q}_\Phi(X)$.

 Now, suppose that $f\in \cD'(X)$ with  $\partial f \in B^{s-1,q}_{\Phi}(X)$. Let 
$\eta(t) := \frac{1}{t} \psi(\frac{1}{2}t) - \frac{1}{t} \psi(4t)$.  
 Then, for every $j\geq 2$ we have $2^j\chi_j(t) \frac{1}{\rho(t)} = -i \eta (2^{-j} t)$ $(t\in \R)$ and 
\begin{align*}
\|2^{sj}\psi_j(\Delta)f\|_{\Phi(X)} & \leq  \| 2^{j} (\chi_j \frac{1}{\rho})(\Delta)\|_{\cL(\Phi(X))} \| 2^{(s-1)j} \psi_j(\Delta) \rho(\Delta) f \|_{\Phi(X)} \\
& = \| \eta(2^{-j}\Delta)\|_{\cL(\Phi(X))} \| 2^{(s-1)j} \psi_j(\Delta) \partial f \|_{\Phi(X)}.
\end{align*}
Therefore, Lemma \ref{fact 2} shows that $f\in B^{s,q}_{\Phi}(X)$. It completes the proof. 
\end{proof}

\begin{remark}\label{norm equiv} (a) The norms $\|\cdot\|_{B^{s,q}_\Phi(X)}$(respectively,  $\|\cdot\|_{F^{s,q}_\Phi(X)}$) of course depend on the chosen function $\psi$ generating the resolution of the identity involved in the definition of the corresponding spaces $B^{s,q}_\Phi(X)$(respectively,  ${F^{s,q}_\Phi}(X)$). However, the arguments presented in the proof of Proposition \ref{per Besov}(ii)(respectively (ii')), show that the norms corresponding to different generating functions are equivalent when the Hardy-Littlewood maximal operator $M_\T$ is bounded on $\Phi$ (respectively, on $\Phi$ and its dual $\Phi'$). The fact that, in the case of Triebel-Lizorkin spaces, one cannot drop, in general, the boundedness of $M_\T$ on $\Phi'$ is well-known from the classical situation, i.e. $\Phi= L^\infty$; cf. \cite[Subsection 2.1.4]{Triebel78}. 

(b) Straightforward arguments show that $B^{-s,1}_\Phi(X^*)$ is separating on $B^{s,\infty}_{\Phi}(X)$ for every $s\in \R$. We apply it in Section \ref{section 5} below. 
\end{remark}

\section{Fourier multipliers}

\subsection{The representation formula for periodic multipliers}\label{per dist}

Let $X$ and $Y$ be Banach spaces.
By the representation result for the periodic distributions $\cD'(\T;X)$ (see, e.g. \cite[Proposition 2.1]{ArBu04}, as well as \cite[Chapter 12]{Ed67} or \cite[Section 3]{ScTr87} for its scalar-valued prototype), it is easily seen that for an arbitrary sequence $m:\ZZ \mapsto \cL(X,Y)$ with a polynomial growth, the following operator 
\begin{equation}\label{represent}
m(\Delta):\cD'(\T,X) \ni f \mapsto \sum_{k\in \ZZ} e_k\otimes m(k)\hat f(k) \in \cD'(\T, Y)
\end{equation}
is well-defined and continuous (the series is convergent in $\cD'(\T; Y)$). 

 Let $\Phi$ be a Banach function space over $(\T,\ud t)$ and 
\[
\E \in \{\Phi, B^{s,q}_\Phi, F^{s,r}_\Phi: s\in \R, q\in [1,\infty], r\in (1,\infty)  \}.
\]
A sequence $m : \ZZ \to \cL(X,Y)$ is called a {\it $\E$-Fourier multiplier} if  $m(\Delta)$ restricts to an operator in $\cL(\E(\T;X), \E(\T; Y))$. We set $\cM_\E(\T; X,Y)$ to denote the space of all $\cL(X,Y)$-valued $\E$-Fourier multipliers, and we simply write $\cM_p(\T; X,Y)$ for $\cM_{L^p}(\T; X,Y)$.

Now, let $\Psi$ be a Banach function space over $(\R,\ud t)$ and 
\[
\E \in \{\Psi, B^{s,q}_\Psi, F^{s,r}_\Psi: s\in \R, q\in [1,\infty], r\in (1,\infty)  \}.
\]
Let $m : \R \to \cL(X,Y)$ be a bounded measurable function and 
\[
D_m := \{f \in \cS'(\R; X) : \cF f \in L^1_{loc}(\R; X) \text{ and } m(\cdot)\cF f \in \cS'(\R; Y)\}
\]
Then, $m$ is called a $\E$-Fourier multiplier, if the following operator
\[
m(D): D_m \cap \E(\R;X)\ni f \mapsto \cF^{-1}( m(\cdot) \cF f) \in \cS'(\R;Y)
\]
extends to an operator in $\cL(\E(\R;X), \E(\R; Y))$. Similarly to the periodic case, we set $\cM_\E(\R; X,Y)$ to denote the space of all $\cL(X,Y)$-valued $\E$-Fourier multipliers, and we simply write $\cM_p(\R; X,Y)$ for $\cM_{L^p}(\R; X,Y)$.

\begin{remark}\label{representation}
Note that in a contrast to the representation formulas for Fourier multipliers on the real line, each periodic multiplier 
has a (canonical) representation formula  (cf. \cite[Section 4]{Kr22}). 
More precisely, since it is not obvious how to define, in general, the extension of the {\it pointwise multiplication} of a function $m \in L^\infty (\R; \cL(X,Y))$ with a distribution $f\in \cS'(\R; X)$, only in special cases one can show that the operator 
\[
m(D): \cS(\R; X)\ni f \mapsto \cF^{-1}( m(\cdot) \cF f) \in \cS'(\R;Y)
\]
extends to a continuous operator on whole $\cS'(\R; X)$. For instance, recall that by means of Schwartz' kernel theorem, such extension holds for each $m\in \cO_M(\R; \cL(X,Y))$, that is, slowly increasing, smooth functions. In this case, we have a formal expression for such extension given for all $f\in \cS'(\R;X)$ by
\[
m(D)f = \cF^{-1} \Theta(m, \cF f), 
\] 
where $\Theta$ is a hypocontinuous, bilinear map from $\cO_M(\R; \cL(X,Y))\times \cS'(\R; X)$ into $\cS'(\R; Y)$; see, e.g. Amann \cite[Theorem 2.1]{Am97}. 
Consequently, compared to the case of the real line, the study of periodic integro-differential equations is conceptually more straightforward at this point (cf. e.g. the presentation of \cite[Proposition 4.2]{Kr22} with Theorem \ref{ext th} below, see also \cite[Problems 3.2 and 3.3]{HyWe06}). 
 \end{remark}

\subsection{The multiplier conditions}\label{sec conditions} For the convenience of the reader we recall the notion of Marcinkiewicz's type conditions used in the next sections.

Let $X$ and $Y$ denote Banach spaces and $\gamma \in \N$. We say that a function $m: \R\rightarrow \cL(X,Y)$ satisfies the {\it Marcinkiewicz condition of order} $\gamma$ (in short, {\it the  $\m^\gamma$-condition}), if $m\in\cC^\gamma(\dot \R; \cL(X,Y))$ and 
\[
[m]_{\m^\gamma}:=\max_{l=0,...,\gamma}\sup_{t\neq 0} \|t^l m^{(l)}(t)\|_{\cL(X,Y)}<\infty \quad \quad  {(\frak M^\gamma)}.
\]
Then, we write $m\in \m^\gamma (\R;\cL(X,Y))$,

Furthermore, let $\widetilde{\frak M}^\gamma(\R;\cL(X,Y))$ denote a subclass of ${\frak M}^\gamma(\R;\cL(X,Y))$, which consists symbols $m$ such that 
\[
[m]_{\widetilde\m^\gamma}:=\max_{l=0,...,\gamma}\sup_{t\neq 0} \|(1+|t|)^l m^{(l)}(t)\|_{\cL(X,Y)}<\infty \quad \quad  {(\widetilde{\frak M}^\gamma)} 
\] 
and $m(0^+)=m(0^-)$ (equivalently, $m$ has continuous extension at $0$).

The discrete counterpart is defined analogously. Namely, a sequence $m : \mathbb{Z} \to \cL(X,Y)$ satisfies {\it the $\m^\gamma$-condition}, if  
\[
[m]_{\m^\gamma}:=\max_{l=0,...,\gamma}\sup_{k\in \mathbb{Z}} \|k^l (\Delta^l m)(k)\|_{\cL(X,Y)}<\infty,
\]
where $(\Delta^lm)(k) = \sum_{j=0}^{l} \binom{l}{j} (-1)^{l-j} m(k+j)$ for $l\in \N_0$. Similarly, we write $m \in \m^\gamma (\mathbb{Z};\cL(X,Y))$.

Furthermore, we say that $m: \ZZ \rightarrow \cL(X,Y)$ satisfies the {\it variational Marcinkiewicz condition}, if 
\[
[m]_{Var}:=\sup_{k\in \ZZ} \|m(k)\|_{\cL(X,Y)} + \sup_{j\geq 0} \sum_{2^j \leq |k| < 2^{j+1}} \|m(k+1) - m(k)\|_{\cL(X,Y)} <\infty.
\]
Set $Var(\ZZ; \cL(X,Y))$ to denote the space of all sequences which satisfy the above condition. Note that $ \m^1(\ZZ; \cL(X,Y)) \subset Var(\ZZ; \cL(X,Y)) $.  

Finally, we say that $m: \ZZ \to \cL(X,Y)$ satisfies the $\m^\gamma_{\cR}$-condition and write $m\in \m^\gamma_\cR(\ZZ; \cL(X,Y))$, if the sets $\{k^l\Delta^l m(k)\}$ are $\cR$-bounded for all $l\in \{ 0,1,...,\gamma\}$. We refer the reader, e.g.  to \cite{KuWe04} or \cite{DeHiPr03a}, for the background on $\cR$-boundedness.

\section{Transference results}\label{transf sec}

In this section we provide some crucial ingredients of the results presented in the following sections. These ingredients allow to transfer directly some results known in the real-line setting to the periodic one, rather than rephrasing the proofs from $\R$ to $\T$. 
Such transference methods reveal some natural questions, which maybe of independent interest; see e.g. Remark \ref{strategy}(a).

\subsection{The operator-valued counterpart of Jodeit's theorem}

The following result is an operator-valued counterpart of an extension of Jodeit's theorem (\cite[Theorem 3.5]{Jo70}) provided by Asmar et al; see\cite[Theorem 1.1]{AsBeGi94}. 

\begin{theorem}\label{Jodeit th} Let $X, Y$ and $Z$ be Banach spaces and  $p\in [1,\infty)$. Let $\Lambda \in L^\infty(\R; Y,Z)$ have a compact support. Then the following assertions are equivalent.
\begin{itemize}
\item [(i)] $\Lambda \in \cM_p(\R; Y,Z)$;

\item [(ii)]  For every $m \in \cM_p(\ZZ; X,Y)$ the function 
\begin{equation}\label{extension formula}
e(\Lambda, m) (t) := \sum_{n\in \ZZ} \Lambda (t - n) m(n)
\end{equation}
is in $\cM_p(\R; X,Z)$. 
\end{itemize}

\end{theorem}

\begin{remark}
Note that (according to the well-known results of Burkholder \cite{Bu81} and Bourgain \cite{Bo83}), if $Y$ is not $U\!M\!D$ space, then not for every compactly supported, scalar function $\lambda\in \cM_p(\R; \C)$, the function $\Lambda (t) := \lambda(t) I_Y$ is in $\cM_p(\R; \cL(Y))$ 
(here we assume that $Y=Z$ and put $I_Y$ for the identity operator on $Y$).
Moreover, since the proof of \cite[Theorem 1.1]{AsBeGi94} is based on several non-trivial ingredients (mainly from \cite{deLe65} and \cite{Jo70}), for the convenience of the reader we outline how the above-stated operator-valued extension of \cite[Theorem 1.1]{AsBeGi94} can be derived from its scalar counterpart. 
\end{remark}

\begin{proof}
Following \cite{AsBeGi94} we split the proof into two steps.
In the first one, we need to show that if $\Lambda$ is a function in $\cM_p(\R; Y,Z)$ with 
$\supp \Lambda \subset [-\frac{1}{4}2^N, \frac{1}{4}2^N ]$ for some $N\in \N$, then the periodic extension of $\Lambda(2^{N} \cdot)$ on $\R$ from $[-1, 1]$, i.e. the function 
$\Lambda^\sharp (t):= \sum_{n\in \ZZ} \Lambda(2^{N}(t- n))$,  is in $\cM_p(\R; Y,Z)$ as well. 

Since, as one can directly check the function $\Lambda(2^{N} \cdot)$ is in $\cM_p(\R; Y,Z)$ with 
\[
\|\Lambda(2^{N}\cdot)\|_{\cM_p(\R; Y,Z)} = \|\Lambda \|_{\cM_p(\R; Y,Z)},
\]
 the proof of the above statement would follow from a direct operator-valued variant of de Leeuw's \cite[Theorem 4.5]{deLe65}. Such variant can be obtained following essentially the lines of the proof of \cite[Theorem 4.5]{deLe65}. 
 For the proof of one of its main ingredients, i.e. \cite[Proposition 4.2]{deLe65}, it seems to be more convenient for a given function $\phi \in \cC_c(\R; X)$ and $N\in \N$, to define $\lambda_N$ as an element of $L^p(\mathbf{D}; X)$ by 
\[
\lambda_N(t):= N^{-1/p} \phi(t/N) \chi_{\frac{1}{N}\ZZ}(t)  \qquad (t\in \mathbf{D}). 
\]
Here, $\mathbf{D}$ stands for the discrete real line $\R$.
Then, if $\mu = (\mu(n)))_{n\in \ZZ}$ is a sequence in $l^p(\ZZ; \cL(X,Y))$ with finitely many non-zero elements, which we consider as an element of $L^p(\mathbf{D}; \cL(X,Y))$, then 
\begin{align*}
\| \mu \star \lambda_N\|^p_{L^p(\mathbf{D}; \cL(X,Y))} & = \sum_{t\in \mathbf{D}} \left| \mu\star \lambda_N (t)\right|^p_Y = \sum_{n\in \ZZ} N^{-1} \left| \sum_{k\in \ZZ} \mu(k)\phi(\frac{n}{N} - k)   \right|^p_Y.
\end{align*}
Moreover, in the context of \cite[Corollary 4.3]{deLe65}, if we set $r(t):=\sum_{k\in \ZZ} \mu(k) e^{ikt}$, $t\in \R$, then 
\[
\cF^{-1}( r(\cdot) \cF \phi) (t) = \sum_{k\in \ZZ} \mu(k)\phi(t-k) \qquad (t\in \R).
\]
Therefore, it is readily seen that 
\[
\|\mu \star \lambda_N\|_{L^p(\mathbf{D}; Y)} \rightarrow \| \cF^{-1}( r(\cdot) \cF \phi)\|_{L^p(\R; Y)} \qquad \textrm{ as } N\rightarrow \infty
\] for every $\cC_c(\R; X)$. 
It gives the desired claim about $\Lambda^\sharp$.

For the second step, we need the following operator-valued counterpart of Jodeit's result, \cite[Theorem 3.5]{Jo70}: \\
{\it Let $\delta(t): = \max(0, 1-|t|)$,  $t\in \R$. Then, 
for arbitrary Banach spaces $X$ and $Y$ and every $m\in \cM_p(\ZZ;X,Y)$ the function 
$e(\delta,m)$ is in $\cM_p(\R;X,Y)$.}
The proof of this statement follows almost verbatim the lines of the proof of its scalar prototype \cite[Theorem 3.5]{Jo70} (see also Remark \ref{rem on Jodeit} below).

Now the implication $(i)\Rightarrow (ii)$, follows from this special variant as in the proof of \cite[Theorem 1.1]{AsBeGi94}. Namely, since Fourier's transform of the de la Vall\'ee Poussin kernel ${V}$ is a linear combination of the dilations of $\delta$, i.e. $\cF V:= 2\delta (2\cdot) - \delta (4\cdot)$,  a direct rescaling argument shows that $e(\cF{{V}}, m)$ is in $\cM_p(\R; X,Y)$ for every $m\in \cM_p(\ZZ; X,Y)$. 
Furthermore, since $e(\Lambda(2^{N}\cdot), m) = \Lambda^\sharp e(\cF{V}, m)$
is in $\cM_p(\R; X,Z)$, again, a rescaling argument yields $(ii)$. 

Finally, note that for $X=Y$ and the sequence $m:\ZZ\rightarrow \cL(Y)$ with $m(0):=I_Y$ and $m(k):=0$ for $k\neq 0$, then $e(\Lambda, m) = \Lambda$. Here, $I_Y$ stands for the identity operator on $Y$. It gives the converse implication, that is $(ii)\Rightarrow (i)$, and completes the proof.  
\end{proof}

\begin{remark}\label{rem on Jodeit}
In the sequel, we apply Theorem \ref{Jodeit th} only in the case when $\Lambda = \lambda I_Y$ and $Y=Z$ for scalar functions $\lambda \in \cC^\infty_c(\R)$. It is natural to reduce the above proof only to its second step, i.e. to a direct proof of such special operator-valued variant of Jodeit's theorem.  
Compared to the proof of \cite[Theorem 3.5]{Jo70}, where $\lambda = \delta$, 
it is not clear how to show  that the sum of $l^1$-norms of Fourier's coefficients of the functions $h_k: \T \rightarrow \C$ given by $h_k (e^{it}) := (\cF^{-1} \lambda) (t+ \pi k)$ for $t\in [-\frac{\pi}{2}, \frac{\pi}{2}]$ and $0$ elsewhere, is finite, i.e. $\sum_{n,k\in \ZZ} |\widehat h_k (n)| < \infty$. 
\end{remark}

For the proof of our extrapolation result, Theorem \ref{ext th} below, we need a refined version of the above operator-valued extension of Jodeit's theorem. Roughly speaking, for a given multiplier sequence $m\in \cM_p(\ZZ; X,Y)$ satisfying  the  Marcinkiewicz condition $\m^3$ on $\T$, we are looking for a {\it summability kernel} $\Lambda=\lambda(\cdot) I_Y$ such that the function $e(\Lambda, m)$ satisfies the corresponding Marcinkiewicz condition $\m^3$ on $\R$. As can be readily seen, not every function $\lambda\in \cC^\infty_c(\R)$ yields such an extension of $m$, and a more suitable choice of $\lambda$ has to be made.

An extension procedure which preserves the first and second order Marcinkiewicz conditions  was provided in \cite[Theorem 5.7.9]{HNVW16} (piecewise affine extension, that is, corresponding to Jodeit's summability kernel $\delta$) and in \cite[Lemma 3.5]{ArBaBu04}, respectively. The following result provides an alternative and unified argument for such extension procedures. 

\begin{theorem}\label{variant of Jodeit's th} Let $X$ and $Y$ be arbitrary Banach spaces.
Then, for every $\gamma \in \{1,2,3\}$ there exists a function $\lambda \in C_c^{\infty}(\R;\C)$  such that for each $m\in \m^\gamma(\ZZ; X,Y)$ the function $e(\lambda,m) :=  \sum_{k\in \ZZ} \lambda (\cdot - k) m(k)$ is an extension of $m$, which belongs to $\m^{\gamma} (\R;\cL(X,Y))$.
\end{theorem}

\begin{proof} 
We give the proof for $\gamma =3$. The proofs for $\gamma = 1$ and $\gamma = 2$ can be readily extracted from that one (see also Remark \ref{rem extension}).

 First, note that for any function $\lambda\in \cC_c^\infty$ with $\supp{\lambda} \subset [-1,3]$ and any sequence $m:\ZZ\rightarrow \cL(X,Y)$, the function $e: =e(\lambda,m)$ can be expressed as follows:
\[
e(s) = m(k_s)\lambda(u_s) + m(k_s-1)\lambda(u_s+1) + m(k_s-2)\lambda(u_s+2) +m(k_s-3)\lambda(u_s+3),
\]
where $k_s := \lfloor{s}\rfloor+1$, $u_s := s-\lfloor{s}\rfloor-1 \in [-1,0]$, and $s\in \R$. 
Moreover,  straightforward arguments show that
\begin{align*}
se'(s) = & s (\Delta m) (k_s-1) \lambda'(u_s) + s (\Delta m) (k_s-2) \Big(\lambda'(u_s)+\lambda'(u_s+1)\Big)\\
& \quad  + s (\Delta m)(k_s-3) \Big(\lambda'(u_s) + \lambda'(u_s+1) + \lambda'(u_s+2)\Big)\\
&  \quad +  s\, m(k_s-3)\Big( \lambda'(u_s)+\lambda'(u_s+1) + \lambda'(u_s+2) + \lambda'(u_s+3)\Big),
\end{align*}
\begin{align*}
s^2e''(s) =   & s^2 (\Delta^2 m) (k_s-2) \lambda''(u_s) + s^2 (\Delta m) (k_s-3)\Big( 2\lambda''(u_s)+\lambda''(u_s+1) \Big)\\ 
& \quad + s^2 m(k_s-2)\Big( 3\lambda''(u_s) + 2\lambda''(u_s+1)+ \lambda''(u_s+2) \Big)\\
& \quad + s^2 m(k_s-3)\Big( -2\lambda''(u_s) -\lambda''(u_s+1)+\lambda''(u_s+3) \Big),
\end{align*} and 
\begin{align*}
s^3e'''(s) =     & s^3 (\Delta^3 m )(k-3) \lambda'''(u_s) + s^3 m (k_s-1) \Big( 3\lambda'''(u_s) +  \lambda'''(u_s+1) \Big)\\
& \quad + s^3 m(k_s-2) \Big( -3\lambda'''(u_s) + \lambda'''(u_s+2) \Big)\\
& \quad +  s^3 m(k_s-3)\Big( \lambda'''(u_s) + \lambda'''(u_s+3) \Big).
\end{align*}
Notice that there exist a constant $c>0$ such that $|s| < c|(k_s - j)|$, $s^2 < c(k_s - j)^2$ and $|s|^3 < c|(k_s - j)|^3$ for $j = 0,1,2$ and $|s|$ sufficiently large. 

Therefore,  if for all $u \in [-1,0]$ the function $\lambda$ satisfied
\begin{equation}\label{lambda 1}
\lambda'(u) + \lambda'(u+1) + \lambda'(u+2) + \lambda'(u+3) = 0,
\end{equation}
 then $m \in  \m^3(\ZZ; \cL(X,Y) ) \subset \m^1(\ZZ; \cL(X,Y))$ would clearly yield the boundedness of $e$ and $(\cdot)e'$.

Similarly, if for all $u \in [-1,0]$ the function $\lambda$ satisfied
\begin{equation}\label{lambda 2}
\begin{cases}
3\lambda''(u) + 2\lambda''(u+1) + \lambda''(u+2) = 0,\\
-2\lambda''(u) -\lambda''(u+1) + \lambda''(u+3) = 0,
\end{cases}
\end{equation}
 then  $m \in \m^3 (\ZZ; \cL(X,Y))\subset \m^2 (\ZZ; \cL(X,Y))$ would imply the boundedness of~$(\cdot)^2 e''$.

Finally,  if for all $u \in [-1,0]$ the function $\lambda$ satisfied
\begin{equation}\label{lambda 3}
\begin{cases}
3\lambda'''(u) + \lambda'''(u+1) = 0,\\
-3\lambda'''(u) + \lambda'''(u+2) = 0,\\
\lambda'''(u) + \lambda'''(u+3) = 0,
\end{cases}
\end{equation}
then $m \in \m^3(\ZZ; \cL(X,Y))$ would give that $(\cdot)^3 e'''$ is bounded.

Notice that the conditions stated in \eqref{lambda 1}, \eqref{lambda 2} and \eqref{lambda 3} are not mutually contradictory. Indeed, one can check that $\eqref{lambda 3}\Rightarrow \eqref{lambda 2}\Rightarrow\eqref{lambda 1}$. Because $e$ must be an extension of $m$, we have to additionally ensure that $\lambda(0) = 1$, $\lambda(1) = 0$, $\lambda(2) = 0$.

To construct a function $\lambda$, which obeys all of the above conditions, on the interval $[-1,0]$ we take $\lambda$ as any smooth function with 
\[\lambda(-1) = 0,\quad \lambda(0) = 1,\quad \lambda'(-1) = 0,\quad \lambda'_{-}(0) = 3/2,\quad \lambda''_{-}(0) = 1,\quad \lambda^{(j)}(k) = 0
\]
for $j \geq 3$ and $k \in \{-1,0\}$. Then, on the interval $[0,3]$, we define $\lambda$ in the following way: 
\begin{align*}
\lambda(u+1) & := -3\lambda(u) + \frac{(u+1)^2}{2} + \frac{3}{2}(u+1) + 1,\\
\lambda(u+2) & := 3\lambda(u) - 2\frac{(u+1)^2}{2} - 2(u+1),\\
\lambda(u+3) & := -\lambda(u) + \frac{(u+1)^2}{2} + \frac{1}{2}(u+1) \qquad (u\in [0,1]). 
\end{align*}
Now it is straightforward to check that for every $u \in [-1,0]$ we have 
\[
\lambda'(u+1) = -3\lambda'(u) + u + \frac{5}{2}, \quad \lambda'(u+2) = 3\lambda'(u) - 2u - 4, \quad \lambda'(u+3) = -\lambda'(u) + u + \frac{3}{2},
\]
\[
\lambda''(u+1) = -3\lambda''(u) + 1, \quad \lambda''(u+2) = 3\lambda''(u) -2, \quad \lambda''(u+3) = -\lambda''(u) + 1,
\]
\[
\lambda'''(u+1) = -3\lambda'''(u), \quad  \lambda'''(u+2) = 3\lambda'''(u),\quad  \lambda'''(u+3) = -\lambda'''(u). 
\]

Using the above expressions for $\lambda'$, $\lambda''$ and $\lambda'''$, it is easy to verify that such function $\lambda$ satisfies \eqref{lambda 1}, \eqref{lambda 2} and \eqref{lambda 3}. The additional conditions imposed on derivatives of $\lambda$ at $0$ ensure that $\lambda$ is smooth. It completes the proof.
\end{proof}

\begin{remark}\label{rem extension} The function $\lambda$ which the existence is proven  above for $\gamma = 3$ works as well for $\gamma = 1$ and $\gamma = 2$. However, note that if we are interested merely in the case when $\gamma = 1$ or $\gamma =2$, then one can simplify the corresponding parts of the above proof by considering $\lambda$ with support in $[-1,1]$ or $[-1,2]$, respectively.
\end{remark}

\subsection{De Leeuw's couples}

Let $\Phi= \Phi(\T)$ and $\Psi= \Psi(\R)$ be two Banach function spaces over $(\T, \ud t)$ and $(\R, \ud t)$, respectively. 
If for every $m\in \cM_\Psi(\R; X,Y)$ such that each $k\in \ZZ$ is a Lebesgue point of $m$ its restriction to $\ZZ$ is in $\cM_\Phi(\ZZ; X,Y)$, then we call $(\Phi, \Psi)$ {\it de Leeuw's couple}. 

The classical de Leeuw theorem \cite{deLe65} (see also its operator-valued counterpart \cite[Theorem 5.93]{HNVW16}) shows that the couples $(L^p(\T;X), L^p(\R;X))$, $p \in [1,\infty)$, have such property. 
Moreover, some further extensions of de Leeuw's result done in the literature provide other examples of such couples; see e.g. \cite{BLVi13} and the references therein. 
In particular, for our further purposes we need the following operator-valued variant of a weighted extension of de Leeuw's theorem due to Berkson and Gillespie \cite[Theorem 1.2]{BeGi03}. 

\begin{lemma}\label{weighted Lp} Let $p\in (1,\infty)$. Let $\widetilde w$ be a $2\pi$-periodic weight and $w(e^{it}):=\widetilde w(t)$ for $t\in\R$. Then, the pair $(L^p_w(\T), L^p_{\widetilde w}(\R))$ is de Leeuw's couple. 

More precisely, for every Banach spaces $X$ and $Y$, if $m\in \cM_{L^p_{\widetilde w}}(\R; X,Y)$   and each point $k\in \ZZ$ is a Lebesgue point of $m$, then 
\[
\|m_{|\ZZ}(\Delta)\|_{\cL(L^p_w(\T;X), L^p_w(\T;Y))}\leq \|m(D)\|_{\cL(L^p_{\widetilde w}(\R;X), L^p_{\widetilde w}(\R;Y))}. 
\]
\end{lemma}

The proof can be obtained by an adaptation of the proof of its scalar counterpart; see e.g. the presentation given in \cite{AnMo09}. Below, we propose a slightly different approach. 

\begin{proof} 
Let $\phi(t):= e^{-\frac{\pi}{p}t^2}$ and $\psi(t):= e^{-\frac{\pi}{q}t^2}$ for $t\in \R$ and $\frac{1}{p} + \frac{1}{q} = 1$.  
As in \cite[Lemma 5.94]{HyNeVeWa17} one can show that
\[
\int_{[-\pi, \pi]}\langle m(\Delta)f (t), g(t)\rangle_{Y,Y^*} \ud t = \lim_{\epsilon\rightarrow 0^+} \epsilon \int_{\R} \langle m(D)(\phi(\epsilon \cdot) f)(t), (\phi(\epsilon \cdot) g)(t) \rangle_{Y,Y^*} \ud t
\]
for all trygonometric polynomials $f:\T \to X,g:\T \to Y^{*}$.
Furthermore, note that 
\begin{align*}
\epsilon \Big|\int_{\R}   \langle m(D)&(\phi(\epsilon \cdot) f)(t),  (\psi(\epsilon \cdot) g)(t) \rangle_{Y,Y^*} \ud t \Big| \\
&  \leq \epsilon \|m(D)(\phi(\epsilon \cdot) f)\|_{L^p_{\widetilde w}(\R;Y)} \|(\psi(\epsilon \cdot) g)\|_{L^q_{\widetilde w^{1-q}}(\R; Y^*)} \\
& \leq \|m(D)\| \epsilon^{\frac{1}{p}}\|\phi(\epsilon \cdot) f)\|_{L^p_{\widetilde w}(\R; X)} \epsilon^{\frac{1}{q}} \|(\psi(\epsilon \cdot) g)\|_{L^q_{\widetilde w^{1-q}}(\R; Y^*)} \\
&  \leq \|m(D)\| \Big(\frac{1}{2\pi}\int_{-\pi}^{\pi} 2 \pi \epsilon \sum_{n \in \ZZ} |\phi(\epsilon (2 \pi n + t))|^p |f(t)|_X^p  w(t) \ud t \Big)^{\frac{1}{p}} \\ 
& \quad \times \Big( \frac{1}{2\pi}\int_{-\pi}^{\pi} 2\pi \epsilon \sum_{n \in \ZZ} |\psi(\epsilon (2 \pi n + t))|^q |g(t)|_{Y^*}^q  w^{1-q}(t) \ud t \Big)^{\frac{1}{q}} \\
&  \leq \|m(D)\| \|\phi\|_{L^p(\R)} \|f\|_{L^p_{w}(X)} \|\psi\|_{L^q(\R)} \|g\|_{L^q_{w^{1-q}}(Y^*)} \\
& \leq \|m(D)\| \|f\|_{L^p_{w}(X)} \|g\|_{L^q_{w^{1-q}}(Y^*)}.
\end{align*}
Since $L^q_{w^{1-q}}(Y^{*})$ is norming for $L^p_{w}(Y)$ we have
\begin{align*}
\|m(\Delta)f\|_{L^p_w(Y)} & = \sup_{\|g\|_{L^q_{w^{1-q}}(Y^{*})} = 1} \quad \int_{[-\pi, \pi]}\langle m(\Delta)f (t), g(t)\rangle_{Y,Y^*} \ud t \\
&  \leq \|m(D)\|\|f\|_{L^p_w(X)}.
\end{align*}
It finishes the proof.
\end{proof}

\begin{remark}\label{strategy} (a) For our purposes (see, e.g. the proof of Theorem \ref{ext th}), for a given Banach function space $\Phi$ over $(\T, \ud t)$ we are interested in constructing a Banach function space $\Psi$ over $(\R, \ud t)$ such that $(\Phi, \Psi)$ is a de Leeuw's couple and, in addition, $\Psi$ inherits some {\it analytic} properties of $\Phi$. For instance, we need to know that the Hardy-Littlewood operator $M_\R$ is bounded on $\Psi$ if $M_\T$ is bounded on $\Phi$. We do not  know if such $\Psi$ can be contracted for Banach function spaces $\Phi$, which are considered in the results of Section \ref{section 5}. Here, we only mention that applying Wiener's amalgam type construction we can define such a space $\Psi$ under additional density and rearrangement type assumptions on $\Phi$. 

It determines our strategy of the proofs of some results presented below (e.g. Theorems \ref{new} and \ref{ext th}), and does not allow to obtain some periodic variants of results already known in the $\R$-setting via a direct transference. 
 
(b) However, by straightforward arguments (see also  \cite[Theorem 2.10]{BoKa97}), one can show that for every Muckenhoupt weight $w\in A_p(\T)$ its periodic extension $\widetilde w$ on $\R$ is in $A_p(\R)$. Therefore, by Muckenhoupt's theorem and Lemma \ref{weighted Lp}, for every $w\in A_p(\T)$ and $p\in (1,\infty)$ the couple $(L^p_w(\T), L^p_{\widetilde w}(\R))$ is de Leeuw's couple and $M_\R$ is bounded on $L^p_{\widetilde w}(\R)$. 
To get desired results for general $\Phi$'s we adopt the Rubio de Francia iteration algorithm to the periodic setting. 
\end{remark}

\subsection{The extrapolation theorem}
The following extrapolation theorem is a crucial ingredient of results presented in the sequel. 

\begin{lemma}\label{RdeF} Let $\Phi$ be a Banach function space over $(\T, \ud t)$ such that the Hardy-Littlewood operator $M_\T$ is bounded on $\Phi$ and its dual $\Phi'$. 
\begin{itemize}
\item [(i)] Let $X$ be a Banach space. Then, for every $p\in (1,\infty)$
\[
\Phi(\T; X) \subset \bigcup_{w\in A_p(\T)}L^p_w(\T; X)
\]
\item [(ii)] Let $X$ and $Y$ be Banach spaces and $p\in (1,\infty)$. Assume that $\{ T_j \}_{j\in J}$ is a family of linear operators  $T_j: \cP(\T; X) \rightarrow\cD'(\T; Y)$ such that for every  $\cW \subset A_p(\T)$ with $\sup_{w\in \cW} [w]_{A_p} < \infty$
\[
\sup_{w\in \cW} \sup_j \| T_j \|_{\cL(L^p_w(\T;X),L^p_w(\T;Y))} < \infty. 
\]
Then, each $T_j$ extends to a linear operator $\cT_j$ on  $ \bigcup_{w\in A_p(\T)} L^p_w(\T;X) \subset L^1(\T;X)$ and has the restriction to an operator in $\cL(\Phi(\T;X), \Phi(\T;Y))$. Moreover, 
\[
\sup_{j\in J} \| \cT_j\|_{\cL(\Phi(\T; X), \Phi(\T;Y))} <\infty.
\] 
\end{itemize}
\end{lemma}
The proof follows the lines of the proof of \cite[Theorem 3.1]{Kr22} almost verbatim. Therefore, we leave it for the reader. 

\begin{remark}\label{RdeF remarks}  
(a) Recall that the Hardy-Littlewood maximal operator $M_\T$ is bounded on $L^p_w(\T)$ for all $p\in (1,\infty)$ and each Muckenhoupt weight $w\in A_p(\T)$; see, e.g. \cite[Theorem 5.2, Corollary 5.3]{BoKa97}.

(b) By the reverse H\"older inequality for weights in $A_p(\T)$ one can show that for every $p\in (1,\infty)$ and $w\in A_p(\T)$ there exists $q>p$ such that 
$L^p_w(\T) \hookrightarrow L^q(\T)$; see the proof of \cite[Theorem 4.1]{KnMcMo16}. 
In the context of a comment stated below \cite[Remark 4.4]{KnMcMo16}, note that for each $p\in (1,\infty)$ 
one can construct a weight $w$ on $[0, 2\pi)$, which is in the class $A_p([0,2\pi])$(as it is defined in \cite{KnMcMo16}), but the function $\T\ni \tau \mapsto w(\arg \tau)$ is not in $A_p(\T)$. 

By Lemma \ref{RdeF}(i) we get that 
\begin{equation}\label{union}
\bigcup \Phi(\T;X) = \bigcup_{p>1} L^p(\T;X),
\end{equation}
 where the first union is taken over all Banach function spaces $\Phi$ over $(\T,\ud t)$ such that $M_\T$ is bounded on $\Phi$ and $\Phi'$.
\end{remark}

We conclude this section with the following consequence of Lemmas \ref{weighted Lp} and \ref{RdeF}. 

\begin{proposition}\label{Hilbert t} Let $X$ be a Banach space with the UMD property.  Then,
for every Banach function space $\Phi$ over $(\T, \ud t)$ such that the Hardy-Littlewood operator $M_\T$ is bounded on $\Phi$ and $\Phi'$, 
 the family 
$\{\chi_I(\Delta): I\subset \R \textrm{ an interval}\}$ is uniformly bounded in $\cL(\Phi(X))$. 
\end{proposition}
\begin{proof}
Recall that the Hilbert transform $H_\R = -i\sgn(D)$ is bounded on $L^2(\R;X)$ and its kernel satisfies the Calder\'on-Zygmund conditions; see, e.g. \cite[Theorem 5.1, p. 374]{HyNeVeWa17}. 
In particular, for each $\cW\subset A_2(\R)$ with $\sup_{w\in \cW} [w]_{A_2}<\infty$, $\sup_{w\in \cW} \|H_\R\|_{\cL(L^2(\R;X))}<\infty$; see, e.g. \cite{RuRuTo86} or \cite{Hy12}. By  Lemma \ref{weighted Lp} we infer that the periodic Hilbert transform $H_\T := -i \sgn (\Delta)$ has the analogues boundedness property as $H_\R$ does. 
Therefore, by Lemma \ref{RdeF} we get that $H_\T$ is in $\cL(\Phi(X))$. 
Since for every $a, b\in \ZZ$ with $a<b$ we have 
\[
\chi_{\{ a, ..., b\}}(k) = \frac{1}{2}\sgn(k - a) - \frac{1}{2}\sgn(k - b) + \chi_{\{a\}}(k) + \chi_{\{b\}}(k) \qquad (k\in \ZZ)
\]
and $\|\sgn(\cdot - n)(\Delta)\|_{\cL(\Phi(X))} = \|H_\T\|_{\cL(\Phi(X))}$ for $n\in \ZZ$, we get the desired claim. 
\end{proof}

\section{The boundedness results for periodic Fourier multipliers} \label{section 5}

This section provides the boundedness results for periodic Fourier multipliers on spaces introduced in Section \ref{spaces}. 
We treat the case of Besov and Triebel-Lizorkin spaces separately to the case of  Banach function spaces $\Phi$. 

In the both subsections, our purpose is to show that the typical multiplier conditions, sufficient for the boundedness on $B^{s,q}_{L^p}$-spaces/$F^{s,q}_{L^p}$-spaces/$L^p$-spaces, ensure the boundedness on $B^{s,q}_\Phi$-spaces/$F^{s,q}_{\Phi}$-spaces/$\Phi$-spaces corresponding to general Banach functions spaces $\Phi$ (see Theorems \ref{new} and \ref{ext th} below).

\subsection{Multipliers on generalized Besov and Triebel-Lizorkin spaces}

In the case of Besov spaces, it is readily seen from the definition of their norms $\|\cdot\|_{B^{s,q}_{\Phi}}$, that the problem whether a sequence $m$ is a multiplier on $B^{s,q}_{\Phi}(X)$ reduces to showing that its {\it dyadic parts} $\psi_j m(\Delta)$, $j\in \N_0$, are uniformly bounded on the underlying Banach function space $\Phi(X)$. 
Note that, in the contrast to the real line case, no extension procedures are needed in the periodic setting; see Subsection \ref{per dist}. 
Moreover, the boundedness of each dyadic part $m\psi_j(\Delta)$ of $m(\Delta)$ is automatic (see Lemma \ref{fact}), and the problem reduces to the uniform boundedness of their norms. For this reason we postpone the study of the boundedness of general multipliers on Banach function spaces $\Phi$ to the next section.  

The following lemma makes the above observation rigorous.

\begin{lemma}\label{aux obs} Let $X$ and $Y$ be Banach spaces. Let $\Phi$ be a Banach function space over $(\T, \ud t)$. Then the following statements hold.

Let $m:\ZZ \rightarrow \cL(X,Y)$ be such that 
\begin{equation}\label{dyadic parts}
\mu:=\sup_{j\in \N_0}\|(\psi_j m)(\Delta)\|_{\cL(\Phi(X), \Phi(Y))}<\infty. 
\end{equation}

Then,  for every $s\in \R$ and $q\in [1,\infty]$, $m\in \cM_{B^{s,q}_\Phi}(\T; X,Y)$. 

More precisely, the multiplier $m(\Delta)$ restricts to an operator in  $\cL(B^{s,q}_\Phi (X), \cL(B^{s,q}_\Phi (Y))$ and for every $f \in \B^{s,q}_{\Phi}(X)$
\begin{equation}\label{convergence}
m(\Delta)f = \sum_{j= 0}^\infty (\psi_j m)(\Delta) f = 
 \lim_{N\rightarrow \infty} \sum_{k\in \ZZ} e_k\otimes \psi(2^{-N}k) m(k) \hat f(k)
\end{equation}
with the convergence in $\B^{s,q}_{\Phi}(Y)$ if $q<\infty$, and in the $B^{-s,1}_{\Phi'}(Y^*)$-topology if $q=\infty$. 

Moreover, if $q=\infty$, then the restriction of $m(\Delta)$ to $\B^{s,\infty}_{\Phi}(X)$ is $\sigma(\B^{s,\infty}_{\Phi}(X), B^{-s,1}_{\Phi'}(X^*))$-to-$\sigma(\B^{s,\infty}_{\Phi}(Y), B^{-s,1}_{\Phi'}(Y^*))$-continuous.
\end{lemma}
\begin{proof}   Since $\psi_j(\Delta) m(\Delta) f = (\chi_j m)(\Delta) \psi_j(\Delta)f$ for every $f\in \cD'(X)$ and $j\in \N_0$, where $\chi_j:=\psi_{j-1}+\psi_j+\psi_{j+1}$ with $\psi_{-1}\equiv 0$, the condition \eqref{dyadic parts} readily gives that $m\in \cM_{B^{s,q}_\Phi}(\T;X,Y)$ for all $s\in \R$ and $q\in [1,\infty]$. 
The point is to show the claimed representation formula for $m(\Delta)$ and its continuity when $q=\infty$.

For $q<\infty$, since $\{\psi_j(\Delta)\}_{j\in \N_0}$ is the resolution of the identity operator on $\cD'(X)$ and, by Lemma \ref{per Besov}, $\cP(X)$ is a dense subset of  $B^{s,q}_\Phi(X)$, it is sufficient to show that the operators $\sum_{0\leq l\leq N} (\psi_l m)(\Delta)$, $N\in \N$, are uniformly in $\cL(B^{s,q}_\Phi(X),B^{s,q}_\Phi(Y))$. Let $f\in B^{s,q}_\Phi(X)$. Then, 
\begin{align*}
\left\| \sum_{0\leq l\leq N} (\psi_l m)(\Delta)f \right\|_{B^{s,q}_\Phi(Y)} & =
 \left(\sum_{j\in \N_0}  \left\| 2^{js} \sum_{0\leq |l-j|\leq 1 \atop 0\leq l \leq N } (\psi_l  m)(\Delta) \psi_j(\Delta) f \right\|^q_{\Phi(Y)}  \right)^{\frac{1}{q}}\\
 & \leq 3 \mu \,\|f\|_{B^{s,q}_\Phi(X)}.
\end{align*}
Let $q=\infty$. We start with the continuity statement. First note that, for all $j\in \N_0$, the adjoint operator to the operator $(\psi_jm)(\Delta) \in \cL(\Phi(X), \Phi(Y))$ on functions $g\in \Phi'(Y^*) \subset [\Phi(Y)]^*$ (see \eqref{duality}) acts as the multiplier $m(-\cdot)^*(\Delta)$. Here, $m(s)^*$ for $s\in \R$ stands for the adjoint of $m(s)\in \cL(X,Y)$. Let 
$f\in B^{s,\infty}_{\Phi}(X)$ and $g\in B^{-s,1}_{\Phi'}(Y^*)$. Since for $j,l\in \N_0$
\[
\bracket{\psi_j(\Delta) g}{\psi_l(\Delta) [m(\Delta) f]}_{\Phi'(Y^*),\Phi(Y)} = 
\bracket{\psi_j(\Delta) [m(-\cdot)^*(\Delta)g]}{\psi_l(\Delta) f}_{\Phi'(X^*),\Phi(X)},
\]
it suffices to show that $m(-\cdot)^*(\Delta)g$ is in $B^{-s,1}_{\Phi'}(X^*)$. For, since 
$\Phi(X)\subset [\Phi'(X^*)]^*$ is a norming subspace of $\Phi'(X^*)$, for each $j\in \N_0$ there exists $h_j\in \Phi(X)$ such that 
\[
\|\psi_j(\Delta) [m(-\cdot)^*(\Delta)g]\|_{\Phi'(X^*)} \leq 
\left|\bracket{\psi_j(\Delta) [m(-\cdot)^*(\Delta)g]}{h_j}_{\Phi'(X^*),\Phi(X)}\right| + 2^{sj - j}.
\]
However,  
\begin{align*}
|\bracket{\psi_j(\Delta)  [m(-\cdot)^*(\Delta)g]}{h_j}_{\Phi'(X^*),\Phi(X)}|  & = \left|\bracket{\psi_j(\Delta) g}{(\chi_j m)(\Delta) h_j}_{\Phi'(X^*),\Phi(X)} \right| \\
& \leq 3 \mu \|\psi_j(\Delta)g\|_{\Phi'(Y^*)}, 
\end{align*}
which gives the desired claim. The convergence of the series in \eqref{convergence} in the 
$B^{-s,1}_{\Phi'}(Y^*)$-topology follows from similar arguments to those presented above. We omit it.  
\end{proof}

In the context of Subsection \ref{per dist}, it is interesting to compare the proof of the above {\it boundedness principle} for periodic Fourier multipliers with the proof of the corresponding result for $\R$, see \cite[Theorem 3.8]{Kr22}; cf. also \cite[Problems 3.2 and~3.3]{HyWe07}.

We complete Lemma \ref{aux obs} with the following supplementary observation on the convergence of the series in \eqref{convergence} under additional assumptions on the geometry of the underlying spaces $X$ and $Y$.

\begin{corollary} Let $X$ and $Y$ be Banach spaces and suppose that $Y$ has the $U\!M\!D$ property. Let $\Phi$ be a Banach function space on $(\T; \ud t)$ such that $M_\T$ is bounded on $\Phi$ and $\Phi'$. Then, for every $q\in [1,\infty)$, $s\in \R$, $m\in \cM_{B^{s,q}_\Phi}(\T; X,Y)$ and $f\in B^{s,q}_\Phi(X)$
\begin{equation}\label{pointw}
m(\Delta)f = \lim_{N\rightarrow \infty} \sum_{|k|\leq N} e_k\otimes m(k)\hat f(k) 
\end{equation}
with the convergence in $B^{s,q}_\Phi(Y)$. 
If, in addition, $Y$ is a Banach function space and $m(\Delta)f\in \Phi(Y)$ (for instance, when $s>0$), then the above convergence holds pointwise almost everywhere on $\T$.
\end{corollary}
\begin{proof}
By Lemma \ref{aux obs} and Proposition \ref{Hilbert t}, we infer that $\chi_{[-N, N]}(\Delta)$, $N\in \N$, are uniformly in $\cL( B^{s,q}_\Phi(Y))$. Since \eqref{pointw} holds for every $f\in \cP(X)$ and $\cP(X)$ is dense in $B^{s,q}_{\Phi}(X)$ (see Lemma \ref{per Besov}), by Lemma \ref{aux obs}, we get \eqref{pointw} for every $f\in B^{s,q}_\Phi(X)$  

For the additional statement, note that by \eqref{union}, $m(\Delta)f \in L^p(Y)$ for some $p>1$. By Rubio de Francia's vector-valued counterpart of Carleson's theorem (see \cite{RdeF86}), we get the pointwise convergence in \eqref{pointw}. It completes the proof. 
\end{proof}
 
\begin{theorem}\label{new} Let $X$ and $Y$ be Banach spaces. Let $\Phi$ denote a Banach function space over $(\T, \ud t)$ such that the Hardy-Littlewood operator $M_\T$ is bounded on $\Phi$. 

Then the following assertions hold.
\begin{itemize}
\item [(i)] For every $\E:=\B^{s,q}_{\Phi}$ with $s\in \R$ and $q\in [1,\infty]$,  we have that 
\[
 \m^2(\ZZ; \cL(X,Y)) \subset \cM_\E(\T; X,Y).
\]

\item [(ii)] 
 Let $X$ and $Y$ have the $U\!M\!D$ property and, in addition, $M_\T$ is bounded on $\Phi'$. Then, for every $\E:=\B^{s,q}_{\Phi}$ with $s\in \R$ and  $q\in [1,\infty]$,  we have that 
\[
 Var(\ZZ;\cL(X,Y)) \subset \cM_\E(\T; X,Y).
\]  

\item [(iii)] If, in addition, $M_\T$ is bounded on $\Phi'$, then the above statements $(i)$ and $(ii)$ hold for every $\E = F^{s,q}_\Phi$ with $s\in \R$ and $q\in (1,\infty)$. 

\end{itemize}
\end{theorem}

\begin{proof} $(i)$ 
Let $\{\psi_j\}_{j\in \N_0}$ be the resolution of the identity on $\R$. 
By Lemma \ref{variant of Jodeit's th} there exists an extension $\widetilde m$ of $m$ on $\R$ such that $\widetilde m\in \m^2(\R; \cL(X,Y))$. Therefore, to get the uniform boundedness of the operators $\psi_j m(\Delta)=(\psi_j\widetilde m)(\Delta)$, $j\in \N$, in $\cL(\Phi(X),\Phi(Y))$, 
 by Lemma \ref{fact 2}, it is sufficient to show that there exists an even, radially decreasing functions $\phi_j$, $j\in \N$, on $\R$ such that $\sup_{j\in \N} \|\phi_j\|_{L^1} <\infty$ and $\|\cF^{-1}(\psi_j \widetilde m)(t) \|_{\cL(X,Y)} \leq \phi_j(t)$ for all $t\in \R$.
The fact that  $(\psi_0 m)(\Delta)=\psi_0\widetilde m(\Delta)$ is in  $\cL(\Phi(\T; X),\Phi(\T; Y))$ follows directly from, e.g. Lemma \ref{fact}.  
  But, the existence of such majorants it is exactly what the proof of \cite[Proposition 4.2(i)]{Kr22} shows ($\phi_j(t)=\frac{C[\widetilde m]_{\m^2} 2^{j}}{1+2^{2j}t^2}$ $(t\in \R)$). Therefore, by Lemma \ref{fact 2} and Lemma \ref{aux obs} we get the desired claim.
 
(ii) First note that for every $j\geq 1$ and $f\in \cD'(X)$ we can write
\begin{align*}
\chi_{[2^{j-1},2^{j+1}]}(\Delta)m(\Delta)&\psi_j(\Delta)f \\
 & = \sum_{l = 2^{j-1}}^{2^{j+1}} e_l \otimes m(2^{j+1})\psi_j(l)\hat{f}(l)\\
 &\quad + \sum_{k = 2^{j-1}}^{2^{j+1}-1} \sum_{l = 2^{j-1}}^{k} e_l \otimes [m(k) - m(k+1)]\psi_j(l)\hat{f}(l) \\ 
& = m(2^{j+1}) \chi_{[2^{j-1},2^{j+1}]}(\Delta)\psi_j(\Delta)f\\
&\quad + \sum_{k = 2^{j-1}}^{2^{j+1}-1} [m(k) - m(k+1)] \chi_{[2^{j-1},k]}(\Delta)\psi_j(\Delta)f 
\end{align*}
Therefore,  
\begin{align*}
&\big\|\chi_{[2^{j-1},2^{j+1}]}(\Delta)\psi_j(\Delta) m(\Delta)f \big\|_{\Phi(Y)} \\
&\quad \leq \left\|m(2^{j+1})\right\|_{\cL(X,Y)} \left\|\chi_{[2^{j-1},2^{j+1}]}(\Delta)\psi_j(\Delta)f \right\|_{\Phi(X)} \\
&\quad \quad + \sup_{2^{j-1} \leq k \leq 2^{j+1}-1} \| \chi_{[2^{j-1},k]}(\Delta)\psi_j(\Delta)f \|_{\Phi(X)} \Big\|\sum_{k = 2^{j-1}}^{2^{j+1}-1} [m(k) - m(k+1)] \Big\|_{\cL(X,Y)}. \\
\end{align*}
Similarly, the analogous estimate holds for $\|\chi_{[-2^{j+1},-2^{j-1}]}(\Delta)m(\Delta)\psi_j(\Delta)f\|_{\Phi(Y)}$. Since $\supp \psi_j\subset \{2^{j-1} \leq |t| \leq 2^{j+1}\}$ ($j\geq 1$),  and $m \in Var(\ZZ;\cL(X,Y))$, by Proposition \ref{Hilbert t}, there exists a constant $\mu>0$ such that for all $f\in \Phi(X)$ and $j\geq 1$ we have
\[
\|\psi_j(\Delta) m(\Delta)f \|_{\Phi(Y)}\leq \mu [m]_{Var} \|\psi_j(\Delta)f \|_{\Phi(Y)}.
\]
Therefore, Lemma \ref{aux obs} completes the proof.

(iii) The proof of this part follows arguments presented in the proof of Proposition \ref{per Besov}. We sketch some details.
First, since $F^{s,q}_{\Psi}(X) = B^{s,q}_\Psi(X)$ for every $\Psi=L^q_w$ ($q\in (1,\infty), w\in A_p(\T)$), $m(\Delta)$ restricts to an operator in $\cL(F^{s,q}_\Psi(X),F^{s,q}_\Psi(Y))$ for such $\Psi's$. Its norm is bounded by 
\[
\mu_{w,q}: = \sup_{j\in \N_0} \|(\psi_jm)(\Delta)\|_{\cL(\Psi(X), \Psi(Y))},
\]
$\sup_{w\in \cW} \mu_{w,q} < \infty$ for each $\cW\subset A_q(\T)$ with $\sup_{w\in \cW}[w]_{A_q}<\infty$. Furthermore, for every $f\in F^{s,q}_{\Psi}(X) $ and $j\in \N_0$
\begin{equation}\label{TL est}
\|\psi_j(\Delta) m(\Delta) f \|_{\Psi(Y)}\leq 3 \mu_{w,q} \|\psi_j(\Delta)f\|_{\Psi(X)}. 
\end{equation}

Fix $\Phi$ and $f\in \cP(X)$. Let  
\[
Gf:= \left( \sum_{j\in \N_0} | 2^{sj} \psi_j(\Delta) f(\cdot) |^q_X \right)^{1/q}. 
\]
and similarly for $G(m(\Delta)f)$. Of course, $Gf$ and $G(m(\Delta)f)$ are in $\Phi$; see,  e.g. \eqref{embedding}. 
Let $h\in \Phi'$,  $h\neq  0$, and set  
\[
w:= w_{Gf, h, q}:= \cR(Gf)^{1-q} \cR' h.
\]
Then $Gf\in L^q_w$, i.e. $f\in F^{s,q}_{L^q_w}(\T; X)$ and $m(\Delta)f \in F^{s,q}_{L^q_w}(\T; X)$. By \eqref{TL est} and a similar argument as in the proof of Proposition \ref{per Besov}(ii) give 
\begin{align*}
3\mu_{w,q}  \|\cR\|_{\cL(\Phi)} \|Gf\|_\Phi \|\cR'\|_{\cL(\Phi')} \| h\|_{\Phi'} 
& \geq 3 \mu_{w,q}  \|\cR Gf\|_\Phi \|\cR'h\|_{\Phi'} \\
& \geq 3\mu_{w,q}  \left( \int_\R \cR(Gf) \cR' h \, \ud t \right)^\frac{1}{q} \left(  \int_\T \cR(Gf) \cR' h \, \ud t \right)^\frac{1}{q'}\\
&\geq  \left( \int_\T G(m(\Delta) f)^q w \,\ud t \right)^\frac{1}{q} \left(  \int_\T \cR(Gf) \cR' h \ud t \right)^\frac{1}{q'}\\
&\geq \int_\T G(m(\Delta) f) h \, \ud t
\end{align*}
Since 
\[
\sup\left\{ [w_{Gf,h,q}]_{A_q}: {f\in \cP(X), h\in \Phi'}  \right\} <\infty,
\]
\[
 \mu:= \sup \left\{ 3 \mu_{w,q}: w=w_{Gf,h,q} \textrm{ with }  f\in \cP(X), h\in \Phi' \right\} <\infty.
\]
Therefore, for every $f\in \cP(X)$
\[
\| m(\Delta) f \|_{F^{s,q}_{\Phi}(\T; Y)}\leq \mu  \|\cR\|_{\cL(\Phi)} \|\cR'\|_{\cL(\Phi')} \| f \|_{F^{s,q}_{\Phi}(X)}.
\]
Since $ \cP(X)$ is dense in $F^{s,q}_{\Phi}(\T;X)$ (see Lemma \ref{per Besov}), $m(\Delta)$ restricts to an operator in $ \cL(F^{s,q}_{\Phi}(X), F^{s,q}_{\Phi}(Y))$. 
\end{proof}

\begin{remark} 
Note that \cite[Theorem 4.2]{ArBu04} shows that the boundedness of $M_\T$ on $\Phi'$ cannot be drop in the part $(ii)$ in general.  
Indeed, $Var(\ZZ; \cL(X))\nsubseteq \cM_{B^{s,\infty}_\Phi}(\T; X,X)$ for $\Phi = L^\infty$ and $s\in (0,1)$, when $X$ is not isomorphic to a Hilbert space.
  
\end{remark}

\subsection{Multipliers on general Banach function spaces}
The main aim of this section is to provide a result on the {\it extrapolation} of the boundedness of periodic Fourier multipliers; see Theorem \ref{ext th}. More precisely, we identify the classes of multipliers $m$ for which a priori knowledge that $m$ is in $\cM_{L^p}(\T; X,Y)$ for some $p\in (1,\infty)$ implies that $m\in \cM_\Phi(\T; X,Y)$ for a large class of Banach function spaces $\Phi$. 
In the next section, we apply this result to prove the phenomenon of the extrapolation of the $L^p$-maximal regularity for a large class of abstract evolution equations; see Theorems \ref{charact of mr}, \ref{charact of wp} and \ref{charact with Z}.
\begin{theorem}\label{ext th} Let $X$ and $Y$ be Banach spaces. 
Let $\Phi$ be a Banach function space over $(\T, \ud t)$ such that $M_\T$ is bounded on $\Phi$ and its dual $\Phi'$.
 Then the following assertions hold.
\begin{itemize}
\item [(i)]  Assume that  $m \in {\m}^3 (\ZZ; \cL(X,Y))$ and $m \in \cM_{L^p}(\T; X,Y) $ for some $p\in (1,\infty)$. Then, $m \in \cM_{\E}(\T; X,Y) $
for every 
\[
\E \in \left\{\Phi,\,B^{s,q}_\Phi,\, F^{s,r}_\Phi :\, s\in \R, q\in [1,\infty], r\in (1,\infty) \right\}.
\]
\item [(ii)] Assume that $m \in {\m}_{\mathcal{R}}^1 (\ZZ; \cL(X,Y))$ and, in addition,  $X$ and $Y$ have the $U\!M\!D$ property. Then, the conclusion of $(i)$ holds. 
\end{itemize}
\end{theorem}

\begin{proof} $(i)$ The case $\E=B^{s,q}_\Phi$ and $\E= F^{s,q}_\Phi$ follows from Theorem \ref{new}. 
Therefore, let $\E=\Phi$. By Theorem \ref{variant of Jodeit's th} there exists an extension $\widetilde m := e(\Lambda, m)$ of $m$ on $\R$ such that $\widetilde m \in \widetilde{\m}^3 (\R; \cL(X,Y))$.  Moreover, Theorem \ref{Jodeit th} yields $\widetilde m \in \cM_{L^p}(\R; X,Y)$. Consequently, $\widetilde m(D)$ is a Calder\'on-Zygmund operator; see \cite[Proposition 4.4.2, p.254]{St93}, or \cite[Lemma 4.1]{Kr22}. In particular, by \cite[Theorem 1.6]{RuRuTo86}, for every $q\in (1,\infty)$ and for every $\widetilde{\cW}\subset A_q(\R)$ with $\sup_{\widetilde{w} \in \widetilde{\cW}} [\widetilde{w}]_{A_q(\R)}<\infty$ we have that 
\[
\sup_{\widetilde{w}\in \widetilde{\cW} }\|\widetilde m (D) \|_{\cL(L^q_{\widetilde{w}}(X), L^q_{\widetilde{w}}(Y))}<\infty.
\]
Let $\cW \subset A_q(\T)$ be such that $\sup_{w \in \cW} [w]_{A_q(\T)}<\infty$. By $\widetilde{w}$ we denote the periodic extension of $w$ on $\R$. Set $\widetilde{\cW} := \{\widetilde{w}: \, w \in \cW\}$. Then simple argumentation shows that $\widetilde{\cW} \subset A_q(\R)$ and there exists a constant $C > 0$ such that
\[
\sup_{\widetilde{w} \in \widetilde{\cW}} [\widetilde{w}]_{A_q(\R)} \leq C \sup_{w \in \cW} [w]_{A_q(\T)} < \infty.
\]
By Lemma \ref{weighted Lp} we get that 
\[
\sup_{w \in \cW} \|m(\Delta)\|_{\cL(L^q_{w}(X), L^q_{w}(Y))} \leq \sup_{\widetilde{w} \in \widetilde{\cW}} \|\widetilde m (D) \|_{\cL(L^q_{\widetilde{w}}(X), L^q_{\widetilde{w}}(Y))} < \infty.
\]
Therefore, by Lemma \ref{RdeF} we conclude that $\widetilde m_{|\ZZ} = m$ is in $\cM_{\Phi}(\T; X,Y)$.

$(ii)$ By Theorem \ref{variant of Jodeit's th} we find $\lambda \in \cC^\infty_c(\R)$ such that $\widetilde m :=e(\Lambda, m)$ (for $\Lambda := \lambda(\cdot)I_Y$) extends $m$ and is in $\m^1(\R; \cL(X,Y))$. By Kahane's contraction principle (see \cite[Proposition 2.5]{KuWe04}), we easily get that $\widetilde m$ satisfies, in fact, the $\m_{\mathcal{R}}^1$-condition. Now \cite[Theorem 3.5.(a)]{FaHyLi20} shows that for every $\cW \subset A_q(\R)$ with $\sup_{w\in \cW}[w]_{A_q(\R)}< \infty$
\[
\sup_{w\in \cW}\|\widetilde m(D)\|_{\cL(L^p_w(\R;X),L^p_w(\R;Y))} <\infty.
\]
Therefore, Lemma \ref{weighted Lp}, Lemma \ref{RdeF} and similar reasoning as in (i) give that $m$ is in $\cM_\Phi(\T; X,Y)$. 
The case of the Besov spaces follows now directly from Lemma \ref{aux obs}. 
For the Triebel-Lizorkin case, note that the functions $\psi_j \widetilde m$ satisfy the $\m^1_{\mathcal{R}}$-condition uniformly in $j\in \N_0$. Therefore, in this case, the proof mimics the arguments presented already in the proof of Theorem \ref{new}(iii). 
\end{proof}

\begin{remark} (a) In the context of applications, see Theorem \ref{mr thm}(c), Theorem \ref{ext th}(i) can be read as an extrapolation of the $L^p$-maximal regularity property. Indeed, as Theorem \ref{conditions} shows, for Fourier multipliers $m$ related to some evolution equations (see Subsection \ref{diff equ}), the fact that $m\in \cM_{L^p}(\R; X,Y)$ implies that $m$ satisfies Marcinkiewicz's condition $(\m^\gamma)$ of an arbitrary order $\gamma \in \N$.   

(b) The assumptions of Theorem \ref{ext th} should be confronted with \cite[Theorem 1]{ArBu04a}, which says that for every Banach space $X$, which is not isomorphic to a Hilbert space, there is a sequence $m: \ZZ  \rightarrow \cL(X)$ satisfying the $\m^\gamma$-condition for every $\gamma\in \N$, but $m$ is not in $\cM_{L^p}(\ZZ; X,X)$ ($p\in (1,\infty)$).  
\end{remark}

\section{Applications}\label{app}
In this section we apply the results developed above to study the solvability and maximal  regularity of an abstract second-order integro-differential equation, see \eqref{aee} below. We start with some preliminaries. 

Let $A,B,P$ be closed operators on a Banach space $X$. By $D_A, D_B, D_P$ we denote their domains equipped with the corresponding graph norms.
Let $\cA$, $\cB$ and $\cP$ denote the {\it evaluations} of operators $A$, $B$ and $P$ on $\cD'(D_A)$, $\cD'(D_B)$ and $\cD'(D_P)$, respectively. More precisely, $\mathcal{A} \in \mathcal{L}(\cD'(D_A),\cD'(X))$ is given by $(\mathcal{A}u)(\phi) = A(u(\phi))$ for every $u \in \cD'(D_A)$ and $\phi \in \cD$. The operators $\cB$ and $\cP$ are defined in the similar manner. 
Moreover, let $c\in \cD'(\cL(Z,X))$, where $Z$ is a Banach space continuously embedded in $X$. 

Let us consider the following abstract degenerated, second-order problem with the convolution term: 
\begin{equation}\label{aee}
\partial \mathcal{P} \partial u + \mathcal{B} \partial u + \mathcal{A}u + c \ast u = f \quad (\text{in } \cD'(X)) \tag{AEE}
\end{equation}
where $f \in \cD'(X)$ is a given $X$-valued distribution. The convolution term $c \ast u$ we interpret as the Fourier multiplier, i.e.  $c\ast u:= \hat c(\Delta) u$.  
For a given $f\in \cD'(X)$, a distribution $u\in \cD'(X)$ is called the {\it distributional solution} of \eqref{aee} if $u\in \cD'(D_A)\cap \cD'(Z)$, $\partial u \in \cD'(D_B)\cap \cD'(D_P)$ and \eqref{aee} holds in $\cD'(X)$. 

 Set 
\begin{equation}\label{space Y}
Y:=D_A\cap D_B \cap D_P \cap Z  
\end{equation}
 with the norm  
\[
 |y|_{Y}:= \max(|y|_A,|y|_B, |y|_P, |y|_Z) \quad (y\in Y).
\]

Our first result shows how the structure of \eqref{aee} affects the relation between the regularities of the symbols of Fourier multipliers, which are involved in the study of  solvability of \eqref{aee}, which we address below. 
To make this result applicable to the different situations (see Section \ref{last}), we need an abstract joint multiplier condition on two symbols. 
For $\gamma \in \N$ we say that a sequence $d:\ZZ\rightarrow \cL(Z,X)$ satisfies the $\m^\gamma$-{\it{condition with respect to a sequence}} $a:\ZZ\rightarrow \cL(X,Z)$, if 
\[
[d]_{\m^\gamma(a)}:= \max_{l=0,..., \gamma} \sup_{k\in \ZZ} \bigl\| k^l (\Delta^l d)(k) a(k+l) \bigr\|_{\cL(X)}<\infty.\quad \tag{$\m^\gamma(a)$}
\]
We write $d\in \m^\gamma(a)$, when the above holds.

\begin{remark}\label{remark about joint conditions}
 Of course, if $d\in \m^\gamma(\ZZ; \cL(Z,X))$ and $a \in l^\infty(\ZZ; \cL(X,Z))$, then $d\in \m^\gamma(a)$. However, under further information on the boundedness of $a$ one can provide more suitable conditions on $d$ to show that $d \in \m^{\gamma}(a)$. For example, if, in addition, we have that $(\cdot)a \in l^{\infty}(\mathcal{L}(X,Z))$, then $(\cdot)^{-1}d \in \m^{\gamma}(\ZZ; \cL(Z,X))$ yields $d\in \m^{\gamma}(a)$. To see it, notice that 
\[
\sup_{l = 0,...,\gamma} \sup_{k \in \ZZ} \|k^{l-1}(\Delta^l d)(k)\|_{\cL(Z,X)} < \infty \quad\textrm{ if and only if } \quad [(\cdot)^{-1}d]_{\m^{\gamma}(\ZZ; \cL(Z,X))} < \infty.
\]
Furthermore, note that $d \in \m^{\gamma - 1}(\ZZ; \cL(Z,X))$ implies $(\cdot)^{-1}d \in \m^{\gamma}(\ZZ; \cL(Z,X))$, but the converse does not hold in general. For instance, for $\gamma = 2$ and the  sequence $d(k) := k$  ($k \in \ZZ$) we have that $(\cdot)^{-1}d \in \m^2$ but $d$ does not satisfy the $\m^{1}$-condition.
We apply these observations below.
\end{remark}
We say that a sequence $d : \ZZ \to \cL(Z,X)$ satisfies the {\it{variational Marcinkiewicz condition with respect to a sequence}} $a:\ZZ\rightarrow \cL(X,Z)$, and write $d\in Var(a)$,  if 
\[
[d]_{Var(a)} := [d]_{\m^0(a)} + \sup_{j \geq 0} \sum_{2^j \leq |k| < 2^{j+1}} \|\Delta d(k)a(k+1)\|_{\cL(X)} < \infty \tag{$\textrm{Var}(a)$}
\]

We say that a sequence $d : \ZZ \to \cL(Z,X)$ satisfies the {\it{$\m^{\gamma}_{\cR}$-condition with respect to a sequence }} $a:\ZZ\rightarrow \cL(X,Z)$ (and write $d\in \m^{\gamma}_{\cR}(a)$), if for each $l = 0,1,...,\gamma$ the set $\{k^l(\Delta^l d)(k)a(k+l)\}_{k\in \ZZ}$ is $\cR$-bounded in $\cL(X)$.

\begin{theorem}\label{conditions} Let $A$, $B$, $P$, $c$ and $Y$ be as stated above. Assume that for every $k\in \ZZ \setminus \{0\}$ the operator 
\[
b(k):= -k^2P + ikB + A + \hat{c}(k)\in \cL(Y,X)
\] is invertible. Let 
\[
a(k):= b(k)^{-1}, \
a_0(k):= kBa(k), \
a_1(k):= k^2Pa(k), \
a_2(k):= ka(k)
\quad (k \in \ZZ \setminus \{0\})
\] 
and $a(0) = a_0(0) = a_1(0) = a_2(0) = 0 \in \cL(X)$.
Then, the following assertions hold.

\begin{itemize}
\item [(i)] Assume that $a\in l^\infty(\cL(X,Z))$ and $a_0, a_1\in l^\infty(\cL(X))$. Then, for every $\gamma \in \{1,2,3\}$, if $\hat{c} \in \m^\gamma(a)$ (respectively, $\hat{c} \in Var(a)$), then $a \in \m^\gamma (\ZZ;\cL(X,Y))$, $a_0, a_1 \in \m^\gamma (\ZZ; \cL(X))$  (respectively, $a\in Var(\ZZ; \cL(X,Y)$,  $a_0, a_1 \in Var(\ZZ; \cL(X))$). 

In addition, if the sequence $a_2 \in l^{\infty}(\cL(X))$, then $a_2\in \m^\gamma (\ZZ; \cL(X))$ (respectively, $a_2 \in Var(\ZZ; \cL(X)$).\\[-1ex]

\item [(ii)] The statement (i) holds in its $\cR$-bound reformulation, that is,  if, in addition, the sequences $a$, $a_0$ and $a_1$ are $\cR$-bounded, 
then for every $\gamma \in \{1,2, 3\}$ the fact that $\hat c$ satisfies the $\m_\cR^\gamma$-condition with respect to a sequence $a$, implies that the sequences $a$, $a_0$ and $a_1$ 
satisfy the $\m_\cR^\gamma$-condition. 

If, in addition,   $a_2$ is $\cR$-bounded then it satisfies the $\m_\cR^\gamma$-condition.
\end{itemize}
\end{theorem}

\begin{proof} First we show that $a\in l^\infty (\cL(X,Y))$. Indeed, note that each condition imposed on $\hat{c}$ implies that $\hat{c}(\cdot)a(\cdot) \in l^\infty(\cL(X))$. 
Since 
\[
I_X:= b(k) a(k) = -k^2Pa(k) + ikBa(k) + Aa(k) + \hat{c}(k)a(k),
\] where $I_X$ denotes the identity operator on $X$, we get $Aa(\cdot) \in l^\infty (\cL(X))$. Moreover, $Z\hookrightarrow X$ implies that $a\in l^\infty (\cL(X))$, which gives our claim that  $a\in l^\infty (\cL(X,Y))$.

(i) First we prove the statement for $a$ and $\gamma \in \{1,2,3\}$. 
For $\gamma = 1$, note that by Leibniz' rule for difference operators  we have
\begin{equation}\label{delta a as delta b 1}
(\Delta a)(k) = -a(k)(\Delta b)(k)a(k+1)\quad (k \in \ZZ),
\end{equation}
and one easily gets that
\begin{equation}\label{delta b}
(\Delta b)(k) = -(2k + 1)P + iB + (\Delta \hat{c})(k) \quad (k \in \ZZ) .
\end{equation}
Thus, by our assumptions on $a_0$, $a_1$ and $\hat c$, 
\begin{equation}\label{b 1}
(k(\Delta b)(k)a(k+1))_{k \in \ZZ} \in \l^{\infty}(\cL(X)).
\end{equation}
Consequently, since $a\in l^\infty(\cL(X,Y))$,
\[
(k(\Delta a)(k))_{k \in \ZZ} \in \l^{\infty}(\cL(X,Y)).
\]
For $\gamma = 2$, since the $\m^2$-condition implies the $\m^1$-condition, it is sufficient to show that 
\[
(k^2(\Delta^2 a)(k))_{k \in \ZZ} \in \l^{\infty}(\cL(X,Y)).
\]
For, note that for all $k\in \ZZ$ we have
\begin{equation}\label{delta a as delta b 2}
(\Delta^2a)(k) = -2(\Delta a)(k)(\Delta b)(k+1)a(k+2) - a(k)(\Delta^2 b)(k)a(k+2). 
\end{equation}
Therefore, by the step for $\gamma = 1$ (see \eqref{b 1}), it is enough to check that  
\begin{equation}\label{b 2}
(k^2(\Delta^2 b)(k)a(k+2))_{k \in \ZZ}\in l^\infty(\cL(X)),
\end{equation}
But this again follows directly from our assumptions, since 
\[
(\Delta^2 b)(k) = -2P + (\Delta^2 \hat{c})(k) \quad (k \in \ZZ).
\]
For $\gamma = 3$, similarly  it suffices to show that
\[
\left(k^3(\Delta^3 a)(k)\right)_{k \in \ZZ} \in \l^{\infty}(\cL(X,Y)).
\]
For, note that for all $k \in \ZZ$ we have
\begin{align}\label{delta a as delta b 3}
(\Delta^3 a)(k) & = -3(\Delta^2 a)(k)(\Delta b)(k+2)a(k+3) - a(k)(\Delta^3 b)(k)a(k+3) \\
\nonumber & \quad -3(\Delta a)(k)(\Delta^2 b)(k+1)a(k+3).
\end{align}
By the steps for $\gamma = 1,2$ (see \eqref{b 1} and \eqref{b 2}), it is sufficient to show the boundedness of 
\[
\left(k^3(\Delta^3 b)(k)a(k+3)\right)_{k \in \ZZ} \subset \cL(X).
\] 
However, since $\Delta^3 b = \Delta^3 \hat{c}$, it follows directly from the assumption on $\hat c$. This finishes the proof of the statement about $a$.

 Now we turn to the sequence $a_0$.
For $\gamma = 1$, note that 
\begin{equation}\label{delta a0}
(\Delta a_0)(k) = iBa(k+1) + ikB(\Delta a)(k).
\end{equation}
Thus, combining \eqref{delta a as delta b 1} and \eqref{b 1} with the boundedness of $a_0$, we get 
\[
\left(k(\Delta a_0)(k)\right)_{k \in \ZZ}\in l^\infty(\cL(X)).
\]
For $\gamma = 2$, note that
\[
(\Delta^2 a_0)(k) = 2iB(\Delta a)(k+1) + ikB(\Delta^2 a)(k) \qquad (k\in \ZZ).
\]
Therefore, the boundedness of 
\[
\left(k^2(\Delta^2 a_0)(k)\right)_{k \in \ZZ}\subset \cL(X)
\]
follows from \eqref{delta a as delta b 2},  \eqref{b 2}, and the assumption on $a_0$.\\
Finally, for $\gamma = 3$, note that 
\[
(\Delta^3 a_0)(k) = 3iB(\Delta^2 a) + ikB(\Delta^3 a)(k) \qquad (k\in \ZZ).
\]
Hence,  the boundedness of 
\[
\left(k^3(\Delta^3 a_0)(k)\right)_{k \in \ZZ}\subset \cL(X)
\]
is implied by the formula \eqref{delta a as delta b 3} and our assumptions on $a_0$ and $\hat{c}$. It completes the proof of the statement about $a_0$.
The fact that $a_1 \in \m^{\gamma}(\ZZ;\cL(X))$ follows from very similar arguments. Therefore, we omit it.\\
Now we assume that $\hat{c} \in Var(\ZZ;a)$. Then, by \eqref{delta a as delta b 1} and \eqref{delta b}, for every $j\in \N$ we have that
\begin{align*}
\sum_{2^j \leq |k| < 2^{j+1}} \|\Delta a(k) \|_{\cL(X,Y)} & \leq \|a\|_{l^{\infty}(\cL(X,Y))} \Big( \sum_{2^j \leq |k| < 2^{j+1}} \frac{1}{|k|} \|k (2k+1)Pa(k+1) \|_{\cL(X)}\\
& \quad +  \sum_{2^j \leq |k| < 2^{j+1}} \frac{1}{|k|} \| k Ba(k+1)\|_{\cL(X)}\\
& \quad + \sum_{2^j \leq |k| < 2^{j+1}} \|(\Delta \hat{c})(k)a(k+1) \|_{\cL(X)} \Big).
\end{align*}
Hence, by the boundedness of $a_0, a_1$ and the condition imposed on $\hat{c}$, we get that $a \in Var(\cL(X,Y))$. To show that $a_0 \in Var(\ZZ;\cL(X))$, using the formulas \eqref{delta a as delta b 1},\eqref{delta b} and \eqref{delta a0}, for every $j\in \N$  we get that
\begin{align*}
\sum_{2^j \leq |k| < 2^{j+1}} &\|\Delta a_0(k) \|_{\cL(X,Y)}\\
 & = \sum_{2^j \leq |k| < 2^{j+1}} \|Ba(k+1) + kB(\Delta a)(k) \|_{\cL(X,Y)}\\ 
& \leq 
 \sum_{2^j \leq |k| < 2^{j+1}} \frac{1}{|k|} \| k Ba(k+1)\|_{\cL(X)}\\
 &\quad  + \|a_0\|_{\l^{\infty}(\cL(X,Y))} \Big( \sum_{2^j \leq |k| < 2^{j+1}} \frac{1}{|k|} \| k (2k+1)Pa(k+1) \|_{\cL(X)}\\
 &\quad  +   \sum_{2^j \leq |k| < 2^{j+1}} \frac{1}{|k|} \| k Ba(k+1)\|_{\cL(X)}\\
 &\quad + \sum_{2^j \leq |k| < 2^{j+1}} \|(\Delta \hat{c})(k)a(k+1) \|_{\cL(X)} \Big).
\end{align*}
Therefore,  the claim about $a_0$ follows from assumptions imposed on $a_0,a_1$ and $\hat{c}$. In a similar manner we prove that $a_1 \in Var(\ZZ; \cL(X))$. The proof of the additional statement about $a_2$ mimics that for $a_0$ (formally, note that $a_2 = a_0$, when $B = I$).

(ii) The proof of the $\cR$-bounded version follows from the same arguments as presented in (i) and the fact that for any Banach spaces $X,Y,Z$, if $\tau,\sigma \subset \cL(X,Y)$ and $\rho \subset \cL(Y,Z)$ are $\cR$-bounded, then the families $\tau + \sigma$ and $\tau \circ \rho$ are $\cR$-bounded.
\end{proof}

\begin{remark}\label{inverse} By the open mapping theorem, the assumption on the operators $b(k)$ made in Theorem \ref{conditions}, i.e. the invertibility of $b(k)$ considered as an operator in $\cL(Y,X)$, is equivalent to say that $b(k)$ considered as an operator on $X$ is bijective and its inverse is in $\cL(X)$. Indeed, the norm $|\cdot|_Y$ is stronger then the graph norm of $b(k)$. 

\end{remark}

 We call a distributional solution of \eqref{aee}, the {\it strong} solution, if $u\in W^{1,1}(X)$ with $u(t)\in D_A$, $u'(t)\in D_B\cap D_P$ for a.e. $t\in \R$, and 
\[
Au,\, Bu',\,  c\ast u \in L^1(X), \textrm{ and } Pu' \in W^{1,1}(X).
\] 

Recall that $W^{1,1}(X) \subset \cC(X)$. Moreover, note that the existence of a strong solution of \eqref{aee} requires that $f\in L^1(X)$, and then \eqref{aee} reads as  
\[
(Pu'(t))' + Bu'(t) + Au(t) + (c\ast u)(t) = f(t) \quad \textrm{ for a.e.  } t\in \R
\]
with $u(0) = u(2\pi)$ and $(Pu')(0) = (Pu')(2\pi)$. The '{\it{prime}}' in the symbols $(Pu'(t))'$ and $Bu'(t)$ refers to the classical derivative, which exists almost everywhere on $\R$ in the topology of $X$. In the context of Theorem \ref{conditions}, note that the multiplier sequence $a$ corresponds to the solution operator of \eqref{aee}, $f\mapsto a(\Delta)f$, the sequences $a_0$ and $a_1$ correspond to the summands $\cB\partial u$ and $\partial \cP \partial u$ in \eqref{aee}, and $a_2$ relates to the strong differentiability of $a(\Delta)f$.

In the following lemma we abstract some facts which allow to adapt multiplier results from the previous section to the study of the solvability of \eqref{aee}.

\begin{lemma}\label{lem to mr} 

(i) Assume that $A$, $B$ and $P$ are closed, linear operators on a Banach spaces $X$ and $c\in\cD'(\cL(Z,X))$, where $Z$ is a Banach space with $Z\hookrightarrow X$. Then, the distributional solution $u$ of \eqref{aee} is a strong one if and only if 
\[
u,\, \partial u,\, \mathcal{P} \partial u ,\, \partial \mathcal{P} \partial u,\, \mathcal{B} \partial u,\,  \mathcal{A}u,\,  c \ast u \in L^1(X).
\]
(ii) Assume that for every function $f=e_k\otimes x$, where $k\in \ZZ$ and $x\in X$, the problem \eqref{aee} has a unique distributional solution.  Then, for every $k\in \ZZ$ the operator $b(k)=-k^2 P + ik B + A + \hat{c}(k)$ is bijective with the bounded inverse.   
\end{lemma}

\begin{proof}
$(i)$ The necessity is readily seen. Combining Lebesgue's differentiation theorem and closedness of operators $A$, $B$ and $P$ it is straightforward to show that this condition is sufficient. For instance, if $u$ is a distributional solution of \eqref{aee}
such that $u\in W^{1,1}(X)$ and the distribution $\cP  \partial u$ is represented by $v \in L^1(X)$, then  for every $\phi \in \cD$
\[
\int_{\T} v \phi \ud t = (\cP \partial u) (\phi) = P \left(\int_\T u' \phi \ud t \right), 
\] and both integrals are convergent in $X$. The Lebesgue differentiation theorem and closedness of $P$ give that for a.e. $t\in [0,2\pi]$, $u'(t)\in D_P$ and $Pu'(t) = v(t)$. Since $v\in L^1(X)$, we get that $Pu' \in L^1(X)$.  
$(ii)$ For the surjectivity, first note that if $u\in\cD'(X)$ is a distributional solution of \eqref{aee} then $\hat u (k)\in D_A\cap Z$ for all $k\in \ZZ$ and $\hat u(k) \in D_B\cap D_P$ for all $k\in \ZZ\setminus \{0\}$. Indeed, note that $\partial u \in \cD'(D_B)$ if and only if $u-\hat u(0)\in \cD'(D_B)$, and similarly for $P$. 
Therefore, if $u$ is the corresponding solution of \eqref{aee} for $f=e_k\otimes x$, where $x\in X$ and $k \neq 0$, then $\hat{u}(k) \in Y$, and in the case when $k=0$, $\hat u (0) \in D_A\cap Z$. Testing the both sides of \eqref{aee} on $e_{-k}\in \cD$ ($k\in \ZZ$) we get that $b(k) \hat u (k) = x$, which yields the suriectivity of $b(k)$.

Suppose now that for some $k\in \ZZ$, $b(k)y = 0 $ for some $y$ in the domain of $b(k)$. 
Then, the function $u:=e_k\otimes y$ satisfies 
\[
\partial \cP \partial u +  \cB \partial u + \cA u + c\ast u = 0.
\]  Therefore, the postulated uniqueness yields $u\equiv 0$, that is, $y=0$. Consequently, the operators $b(k)$, $k\in \ZZ$, are bijective.

Finally, since $b(k)\in \cL(Y,X)$ for $k\neq 0$, and  $b(0)\in\cL(D_A\cap Z, X)$, and $Y, D_A\cap Z \hookrightarrow X$, the boundedness of the inverse of $b(k)$ on $X$ follows from the open mapping theorem.  It finishes the proof of $(ii)$.    
\end{proof}

For $\E\in \{\Phi, B^{s,q}_\Phi, F^{s,q}_\Phi: \Phi \textrm{ a Banach function space over } (\T, \ud t) \}$,
we say that the problem \eqref{aee} has $\E$-{\it{maximal regularity}}, if for every $f\in \E(X)$ there exists a unique distributional solution $u$ of \eqref{aee} such that 
\[
u-\hat u(0)\in \E(Y) \quad \textrm{and} \quad \partial \mathcal{P} \partial u, \, \mathcal{B} \partial u,\, \mathcal{A}u, \,  c \ast u  \in \E(X). 
\]

Moreover, for $\E\subset L^1$, we say that \eqref{aee} is $\E$-{\it{well-posed}}, if for every $f\in \E(X)$ the problem \eqref{aee} has a unique strong solution such that 
\begin{equation}\label{strong}
u,\, u',\, Pu',\, (Pu')',\, Bu',\,  Au,\, c\ast u \in \E(X).
\end{equation}

Note that by Lemma \ref{lem to mr}(i), if $\E\subset L^1$ and \eqref{aee} has $\E$-maximal regularity, then \eqref{aee} is $\E$-well-posed if and only if 
\[
u,\, \partial u,\, \cP \partial u \in \E(X). 
\]

The following result is the main result of this section. In the points (a) and (b) we address the questions of the maximal regularity and well-posedness of the problem \eqref{aee} under different assumptions on the geometry of the underlying Banach space $X$, as well as multiplier conditions imposed on corresponding multiplier symbols. 
The point (c) clarifies the phenomenon of {\it extrapolation} of $L^p$-maximal regularity and $L^p$-well-posedness for such problem.    

For the simplicity of its formulation, let $\cL_M$ denote the family of all Banach function spaces over $(\T, \ud t)$ on which the Hardy-Littlewood operator $M_\T$ is bounded,~{i.e.}
\[
\cL_M :=  \{ \Phi \textrm{ a Banach space over }(\T, \ud t): M_\T \textrm{ is bounded on }\Phi\}.
\]

\begin{theorem}\label{mr thm} Let $X$ and $Z$ be Banach spaces such that $Z\hookrightarrow X$. Let $A$, $B$ and $P$ be closed, linear operators on a Banach space $X$,  $c\in \cD'(\T; \cL(Z,X))$ and $Y$ has the meaning specified in \eqref{space Y}. For every $k\in \ZZ$ let
\[
b(k) := -k^2 P + ik B + A + \hat{c}(k). 
\]

\emph{(a)} Assume that for every $k\in \ZZ$ the operator $b(k)$ is bijective and 
 $a\in l^\infty(\cL(X,Z))$ and $a_0, a_1\in l^\infty(\cL(X))$, where 
 \[
a(k):= b(k)^{-1}, \quad a_0(k):= ikBa(k), \quad a_1(k):= -k^2 Pa(k)\quad (k\in \ZZ \setminus \{0\}).
\] 
and $a(0) = a_0(0) = a_1(0) = 0$.
Then, the following assertions hold.\\[-1ex] 
\begin{itemize}
\item [(a1)] If $\hat{c} \in \m^2(\ZZ; a) $, then for every 
\[
\E \in \bigl\{ B^{s,q}_\Psi, F^{s,r}_\Phi: s\in \R, q\in [1,\infty], r\in (1,\infty),\,\, \Phi,\Phi', \Psi \in \cL_M \bigr\}. 
\]
the problem \eqref{aee} has the $\E$-maximal regularity. 
In addition, if the sequence $(ka(k))_{k\in \ZZ}$ is bounded in $\cL(X)$ and $\E\subset L^1$, then \eqref{aee} is $\E$-well-posed.\\[-1ex] 

\item [(a2)] If $\hat{c} \in Var(\ZZ;a)$ and, in addition, $X$ has $U\!M\!D$-property, then the conclusion of \emph{(a1)} holds for each 
\[
\E\in \bigl\{ B^{s,q}_\Phi, \,F^{s,r}_\Phi: s\in \R, q\in [1,\infty], r\in (1,\infty),\,\, \Phi, \Phi' \in \cL_M   \bigr\}.   
\] 
\end{itemize} 

\emph{(b)} Assume that for every $k\in \ZZ$ the operator $b(k)$ is bijective, the sequences $a$, and $a_0$, $a_1$ are $\cR$-bounded in $\cL(X,Z)$ and $\cL(X)$, respectively,  and $X$ has the $U\!M\!D$ property. 

If $\hat{c} \in \m_\cR^1(\ZZ; a)$, then for every 
\[
\E \in \bigl\{ \Phi, \,B^{s,q}_\Phi,\, F^{s,r}_\Phi: s\in \R, q\in [1,\infty], r\in (1,\infty), \,\, \Phi, \Phi' \in \cL_M \bigr\}
\]  the problem \eqref{aee} has $\E$-maximal regularity.\\[-2ex] 

In addition, if $(ka(k))_{k\in \ZZ}$ is $\cR$-bounded in $\cL(X)$, then for every $\E\subset L^1$, the problem \eqref{aee} is $\E$-well-posed. 

In particular, the last statement holds for each
\[
\E \in \bigl\{ \Phi,\, B^{s,q}_\Phi,\, F^{s,r}_\Phi: s>0, q\in [1,\infty], r\in (1,\infty),\,\, \Phi, \Phi' \in \cL_M \bigr\}.
\]

\emph{(c)} 
Assume that the problem \eqref{aee} has $L^p$-maximal regularity for some $p\in (1,\infty)$ (respectively, \eqref{aee} is $L^p$-well-posed).

If $\hat{c} \in \m^3(\ZZ; a)$, then for every 
\[
\E \in \bigl\{ \Phi, \,B^{s,q}_\Psi,\, F^{s,q}_\Phi: s\in \R, q\in [1,\infty], r\in (1,\infty), \,\, \Phi, \Phi', \Psi \in \cL_M \bigr\}
\] it has $\E$-maximal regularity (respectively, in addition, if $\E\subset L^1$, it is $\E$-well-posed).

\end{theorem}

\begin{proof} (a1) Combining Theorem \ref{conditions} with Theorem \ref{new}(i) and (iii)  we infer that $a\in \cM_\E (\T; X,Y)$ and $a_0, a_1, Aa(\cdot) \in \cM_\E(\T; X,X)$.
 In particular,  for every $f\in \E(X)$, if we put $u:=a(\Delta) f + b(0)^{-1} \hat f(0)$, then $u \in \cD'(D_A)\cap \cD'(Z)$ and $\partial u \in \cD'(Y)\subset \cD'(D_B)\cap \cD'(D_P)$. It is easy to check that $u$ is a distributional solution of \eqref{aee}. Since the Fourier coefficients of distributions are uniquely determined, the injectivity of the operators $b(k)$, $k\in \ZZ$, gives that $u$ is the unique solution of \eqref{aee}. 

Moreover, note that
\[
\partial \cP \partial u = a_1(\Delta) f, \, \cB \partial u = a_0(\Delta)f, \, \cA u = [A a(\cdot)](\Delta) f \in \E(X)
\]
and, since $f\in \E(X)$, 
\[
 c\ast u = \hat{c}(\Delta)a(\Delta)f = \partial \cP \partial u + \cB \partial u + \cA u - f \in \E(X). 
 \]
It proves the $\E$-maximal regularity of \eqref{aee}. 

For the additional statement, again by Theorem \ref{conditions} and Theorem \ref{new}(i), we get that $a_2(k):= ik a(k)$ and $a_3(k):= ik Pa(k)$, $k\in \ZZ$, are in $\cM_\E(\T; X,X)$. Since $Y\hookrightarrow X$ and $D_A\cap Z\hookrightarrow X$, we have that $a\in \cM_\E(\T; X,X)$. Hence, if $f\in \E(X)$, then for the corresponding solution $u = a(\Delta) f + b(0)^{-1}\hat{f}(0)$ of \eqref{aee} we get that
\[
u,\, \partial u = a_2(\Delta)f,\, \cP\partial u = a_3(\Delta)f  \in \E(X).
\]
If, in addition, $\E\subset L^1$, Lemma \ref{lem to mr}(i) shows that $u$ is a strong solution and \eqref{strong}, that is, the $\E$-well-posedness of \eqref{aee}. 
It completes the proof of (a1).

Relying on Theorem \ref{conditions}, the parts (ii) and (iii) of Theorem \ref{new}, and Lemma \ref{lem to mr}(i), the proof of the statement (a2) follows the similar arguments presented for the proof of (a1). Therefore, we omit them. 

For the statement (b), note that Theorem \ref{ext th}(ii) reduces the proof of (b) to the arguments provided in the proof of (a1). 

Finally, for the proof of (c), by Lemma \ref{lem to mr}(ii), the $L^p$-maximal regularity of \eqref{aee} implies that for each $k\in \ZZ$, $b(k)$ is bijective with a bounded inverse. That is, $a(k) = b(k)^{-1} \in \cL(X,Y)$ for all $k\in \ZZ\setminus \{0\}$ and $a(0) = b(0)^{-1}\in \cL(X,D_A\cap Z)$ (see Remark \ref{inverse}(i)). It is straighforward to check that the corresponding solution operator for \eqref{aee} is given by the Fourier multiplier $a(\Delta)$. 

We show that  $a\in l^\infty(\cL(X,Y))$ and $a_0$, $a_1 \in l^\infty(\cL(X))$.  The $L^p$-maximal regularity implies that the maps 
\[
L^p(X) \ni f \mapsto a(\Delta) f \in L^p(Y) \textrm{ and } L^p(X)\ni f \mapsto \cB \partial a(\Delta) f,\,\,\partial \cP \partial a(\Delta) f \in L^p(X),
\] 
 are well-defined and (by the uniqueness) linear. By the closed graph theorem, it is straightforward to see that these maps are bounded. Moreover, since $a(\Delta) (e_k\otimes x) = e_k\otimes b(k)^{-1}x$ for every $x\in X$ and $k\in \ZZ$, we infer that for $k\neq 0$
\[
|b^{-1}(k)x|_{Y} =\| a(\Delta) (e_k\otimes x)\|_{L^p(Y)} \leq \| a(\Delta)\|_{\cL(L^p(X),L^p(Y))} |x|_X   
\]
and similarly
\[
|kB b(k)^{-1}x|_X \leq \|\cB \partial a(\Delta)\|_{\cL(L^p(X))}|x|_X, \,\,
|k^2P b(k)^{-1}x|_X \leq \|\partial \cP \partial a(\Delta)\|_{\cL(L^p(X))}|x|_X. 
\]
It gives our claim. Combining it with the assumption on $\hat{c}$, by Theorem \ref{conditions} and Theorem \ref{ext th}(i), we are in the position to use the same arguments as in (a1) to get the desired assertion on the $\E$-maximal regularity of \eqref{aee}. 
For the additional statement, similarly as above, the $L^p$-well-posedness of \eqref{aee} implies that the sequences $a_2(k) = ika(k)$ and $a_3(k) = ikPa(k)$ ($k\in \ZZ$) are in $\cM_\E(\T; X,X)$. Hence, the proof follows the corresponding lines of the proof of (a1). This completes the proof of (c). 
\end{proof}

\begin{remark}
Since for each $\frak M \in \{ \m^\gamma, \m^\gamma_\cR, Var\}$, if $(k B a(k))_{k\in \ZZ}, (k^2 P a(k))_{k\in \ZZ}$ are in $\frak M(\ZZ, \cL(X))$ then  
$(B a(k))_{k\in \ZZ}, (k P a(k))_{k\in \ZZ}, (P a(k))_{k\in \ZZ} \in \m(\ZZ; \cL(X))$, the proof of Theorem \ref{mr thm} shows that the $\E$-maximal regularity of \eqref{aee} concluded in each of its statements (a), (b), (c) (no additional assumptions on $(ka(k))_{k\in \ZZ}$ is required) gives that the problem 
\[
\partial^2 \cP u + \partial\cB u + \cA u + c\ast u = f 
\] is $\E$-well-posed whenever $\E\subset L^1$. That is, for each $f\in \E(X)$ there exists a unique distributional solution $u$ such that  $u-\hat u(0) \in \E(Y)$, $Pu\in W^{2,1}(X)\subset \cC^1(X)$, $Bu\in W^{1,1}(X)\subset \cC(X)$ and 
\[
(Pu)''(t) + (Bu)' + Au(t) + c\ast u(t) = f(t)\qquad \textrm{ a.e. } t \in [0,2 \pi].
\]   
In this context note that if $\partial^2 \cP u, \partial \cB u \in B^{s,q}_\Phi(X)\subset \Phi(X)$ for some $s>0$, then by Proposition \ref{per Besov}(iii), we immediately get $Pu \in W^{2,1}(X)$ and $Bu\in W^{1,1}(X)$.
\end{remark}

As was already mentioned in Remark \ref{remark about joint conditions} one can impose separate conditions on $a$ and $\hat c$ which imply their joint condition $\hat c \in \m^\gamma (a)$ (or others). It leads to the following characterisation of the maximal regularity and well-posedness of the problem \eqref{aee}. Here, we assume that $Z=X$; see Remark \ref{rem on Z} below. We start with the characterization and extrapolation of the maximal regularity of \eqref{aee}.

\begin{theorem}\label{charact of mr} Let $A$, $B$ and $P$ be closed, linear operators on a Banach space $X$. Let $c\in \cD'(\cL(X))$.

\noindent \emph{(a)} Assume that  $\hat c  \in \m^2(\ZZ;\cL(X))$. 
 Then, the following assertions are equivalent.

\begin{itemize}
\item [(i)] The problem \eqref{aee} has $\E$-maximal regularity for every $\E$ such that  
\[
\E \in \bigl\{ B^{s,q}_\Psi, F^{s,r}_\Phi: s\in \R,  q\in [1,\infty], r\in (1,\infty),\,\, \Phi,\Phi', \Psi \in \cL_M \bigr\}.
\]

\item [(ii)] The problem \eqref{aee} has $B^{s,q}_{L^p}$-maximal regularity for some $p\in (1,\infty)$ and $q\in [1,\infty]$.

\item [(iii)] For every $k\in \ZZ$ the operator $b(k):=-k^2P + ik B + A + \hat c(k)$ is bijective and the sequences 
\[
\left( b(k)^{-1}\right)_{k\in \ZZ}, \quad (k B b(k)^{-1})_{k\in \ZZ}, \quad (k^2P b(k)^{-1})_{k\in \ZZ}
\]
are bounded in $\cL(X)$.\\
\end{itemize}

\noindent \emph{(b)} Assume that $X$ has the $U\!M\!D$ property and  $\hat c\in \m^1_\cR(\ZZ;\cL(X))$.
 Then, the following assertions are equivalent.

\begin{itemize}
\item [(i)] The problem \eqref{aee} has $\E$-maximal regularity for every $\E$ such that  
\[
\E \in \bigl\{ \Phi,\, B^{s,q}_\Psi,\, F^{s,r}_\Phi: s\in \R,  q\in [1,\infty], r\in (1,\infty),\,\, \Phi, \Phi', \Psi \in \cL_{M} \bigr\}.
\]

\item [(ii)] The problem \eqref{aee} has $L^p$-maximal regularity for some $p\in (1,\infty)$.
 
\item [(iii)] For every $k\in \ZZ$ the operator $b(k):=-k^2P + ik B + A + \hat c(k)$ is bijective and the sequences 
\[
(b(k)^{-1})_{k\in \ZZ}, \quad (k B b(k)^{-1})_{k\in \ZZ}, \quad (k^2P b(k)^{-1})_{k\in \ZZ}
\]
are $\cR$-bounded in $\cL(X)$ .
\end{itemize}
\end{theorem}

\begin{remark}\label{rem on charact mr}
(a)  It is easily seen that for \eqref{aee} with $B=I \in\cL(X)$ one can relax the assumption on $\hat c$ in the statement (a) and (b) of Theorem \ref{charact of mr} to $(\cdot)^{-1}\hat c \in \m^2(\ZZ;\cL(X))$ in (a), and to $(\cdot)^{-1}\hat c \in \m^1_\cR(\ZZ;\cL(X))$ in (b), respectively. 

Furthermore, in the case when $P=I\in \cL(X)$, one can further weaken the assumption on $\hat c$, namely, to $(\cdot)^{-2}\hat c \in \m^2(\ZZ;\cL(X))$ in (a), and to $(\cdot)^{-2}\hat c \in \m^1_\cR(\ZZ;\cL(X))$ in (b), respectively; cf. also Remark \ref{remark about joint conditions}.

 (b) The assertion (ii) in Theorem \ref{charact of mr}(a) can be replaced with the following one 
{\it
\begin{itemize}
\item [(ii')] The problem \eqref{aee} has $\E$-maximal regularity for some $\E$ such that
\[
\E \in \bigl\{B^{s,q}_\Psi,\, F^{s,r}_\Phi: s\in \R,  q\in [1,\infty], r\in (1,\infty),\,\, \Phi,\Phi', \Psi \in \cL_M \bigr\}.
\]
\end{itemize}  }
Indeed, one can easily check that the maximal regularity of \eqref{aee} with respect to every such $\E$ implies the boundedness of the sequences stated (iii). It is not clear if $\E$-maximal regularity of \eqref{aee} for general $\E$ yields the $\cR$-boundedness of these sequences.

(c) In a similar manner to that in Theorem \ref{charact of mr} one can formulate the statement corresponding to $(a2)$ of Theorem \ref{mr thm}. Its proof follows the same arguments. We leave it to the interested reader. 

\end{remark}

 \begin{proof}[Proof of Theorem \ref{charact of mr}]
(a) Of course (i)$\Rightarrow$(ii). For (ii)$\Rightarrow$(iii), since $e_k\otimes x \in B^{s,q}_{L^p}(X)$ for every $k\in \ZZ$ and $x\in X$ (see, e.g. Proposition \ref{per Besov}(i)), by Lemma \ref{lem to mr}(ii) we get that for all $k\in \ZZ$ the operator $b(k)$ is bijective with inverse in $\cL(X,Y)$ when $k\neq 0$, and in $\cL(X, Z)$ when $k=0$. 
Analogously as in the proof of Theorem \ref{mr thm}(c), by the closed graph theorem, we infer that the maps 
\[
B^{s,q}_{L^p}(X)\ni f \mapsto a(\Delta) f \in B^{s,q}_{L^p}(Z)  \textrm{ and }
\]
\[
B^{s,q}_{L^p}(X)\ni f \mapsto \cB \partial a(\Delta) f,\,\, \partial \cP \partial a(\Delta) f \in B^{s,q}_{L^p}(X)
\]
are bounded, where $a(k):=b^{-1}(k)$, $k\in \ZZ$. 
Since for every $k\in \ZZ$ and $x\in X$
\[
\|e_k\otimes x\|_{B^{s,q}_{L^p}(X)} = \Big( \sum_{j\geq 0} 2^{sjq} \psi_j(k)^q \Big)^{\frac{1}{q}} |x|_X, 
\] we get the boundedness of desired sequences in (iii). The proof of (iii)$\Rightarrow$(i) follows directly from Theorem \ref{mr thm}(a1); see Remark \ref{remark about joint conditions}.

The proof of (b) mimics the same arguments, we only recall here that each sequence in $\cM_p(\ZZ, X,Y)$ (for arbitrary Banach spaces $X$ and $Y$) is necessarily $\cR$-bounded; see \cite{ClPr01} or \cite{ArBu02}.
\end{proof}

Now we formulate analogous result for the well-posedness of \eqref{aee}.  Its proof follows the lines of the proof of Theorem \ref{charact of mr} with straightforward modifications. Therefore, we omit it.

\begin{theorem}\label{charact of wp} Let $A$, $B$ and $P$ be closed, linear operators on a Banach space $X$. Let $c\in \cD'(\cL(X))$.

\noindent\emph{(a)} Assume that  $( k^{-1}\hat c(k))_{k\in \ZZ} \in \m^2(\ZZ;\cL(X))$. 
Then, the following assertions are equivalent.

\begin{itemize}

\item [(i)] The problem \eqref{aee} is $\E$-well-posed for every $\E\subset L^1$ such that  
\[
\E \in \bigl\{ B^{s,q}_\Psi, F^{s,r}_\Phi: s\in \R,  q\in [1,\infty], r\in (1,\infty),\,\, \Phi,\Phi', \Psi \in \cL_M \bigr\}
\]

\item [(ii)] The problem \eqref{aee} is $B^{s,q}_{L^p}$-well-posed for some $p\in (1,\infty)$ and $q\in [1,\infty]$.

\item [(iii)] For every $k\in \ZZ$ the operator $b(k):=-k^2P + ik B + A + \hat c(k)$ is bijective and the sequences 
\[
\left(k b(k)^{-1}\right)_{k\in \ZZ}, \,\, (k B b(k)^{-1})_{k\in \ZZ}, \,\, (k^2P b(k)^{-1})_{k\in \ZZ}
\]
are bounded in $\cL(X)$.
\end{itemize}
\noindent\emph{(b)}  Assume that $X$ has the $U\!M\!D$ property and  $( k^{-1} \hat c (k))_{k\in\ZZ} \in \m^1_\cR(\ZZ;\cL(X))$. Then, the following assertions are equivalent.

\begin{itemize}

\item [(i)] The problem \eqref{aee} is $\E$-well-posed for every $\E$ such that  
\[
\E \in \bigl\{ \Phi,\, B^{s,q}_\Psi,\, F^{s,r}_\Phi: s\in \R,  q\in [1,\infty], r\in (1,\infty),\,\, \Phi, \Phi', \Psi \in \cL_{M} \bigr\}.
\]

\item [(ii)] The problem \eqref{aee}  is $L^p$-well-posed for some $p\in (1,\infty)$.

\item [(iii)] For every $k\in \ZZ$ the operator $b(k):=-k^2 P + ik B + A + \hat c(k)$ is bijective and the sequences 
\[
(k b(k)^{-1})_{k\in \ZZ}, \quad (k B b(k)^{-1})_{k\in \ZZ}, \quad (k^2P b(k)^{-1})_{k\in \ZZ}
\]
are $\cR$-bounded in $\cL(X)$. 
\end{itemize}
\end{theorem}

\begin{remark}\label{rem on charact of wp}
The similar statements to those stated in Remark \ref{rem on charact mr} hold in the context of Theorem \ref{charact of wp}. In particular, for \eqref{aee} with $P=I \in\cL(X)$ in \eqref{aee}, one can relax the assumption on $\hat c$ in the statements (a) and (b) of Theorem \ref{charact of wp} to $(\cdot)^{-2}\hat c \in \m^2(\ZZ;\cL(X))$ in (a) and to $(\cdot)^{-2}\hat c \in \m^1_\cR(\ZZ;\cL(X))$ in (b), respectively.

It is worth noticing that according to Remark \ref{remark about joint conditions},  if $\hat{c} \in \m^1(\ZZ;\cL(X))$, then $(\cdot)^{-1}\hat{c} \in \m^2(\ZZ;\cL(X))$. Thus, for $c$ such that $\hat{c} \in \m^1(\ZZ;\cL(X))$ the equivalence in Theorem \ref{charact of wp}(a) also holds. This is a strictly stronger condition, but in some situations it might be easier to verify.
\end{remark}

\begin{remark}\label{rem on Z}
If we modify the notion of the well-posedness of \eqref{aee} (which sometimes can be motivated by a special form of the convolution therm; see, e.g. \eqref{delay op}) by replacing the condition $u\in W^{1,1}(X)$ by a stronger one, e.g. $u\in W^{1,1}(Z)$, where $Z$ is a Banach space such that $Z \hookrightarrow X$, then the above characterizations, Theorems \ref{charact of mr} and \ref{charact of wp} (where $Z=X$) can be easily adjusted to such modified setting. The proofs of such modifications have the same pattern as those of Theorems \ref{charact of mr} and \ref{charact of wp}. Here, we do not provide such reformulation of the above characterization results, which corresponds to such stronger notion of well-posedness of \eqref{aee}.

For our further purposes we present a counterpart of Theorems \ref{charact of mr} and \ref{charact of wp} for a convolutor $c\in \cD'(\cL(Z,X))$ with and arbitrary $Z\hookrightarrow X$, which we apply in the next section.   
\end{remark}

\begin{theorem}\label{charact with Z} Let $A$, $B$ and $P$ be closed, linear operators on a Banach space $X$. Let $c\in \cD'(\cL(Z,X))$, where $Z$ is a Banach space 
continuously embedded in $X$.

\emph{(a)} Assume that  $\hat c \in \m^2(\ZZ;\cL(Z,X))$. Then, the following assertions are equivalent.

\begin{itemize}
\item [(i)] The problem \eqref{aee} has $\E$-maximal regularity (respectively, is $\E$-well-posed) for every $\E$ such that  
\[
\E \in \bigl\{ B^{s,q}_\Psi, F^{s,r}_\Phi: s\in \R,  q\in [1,\infty], r\in (1,\infty),\,\, \Phi,\Phi', \Psi \in \cL_M \bigr\}
\]
(respectively, in addition, $\E\subset L^1$).

\item [(ii)] The problem \eqref{aee} has $B^{s,q}_{L^p}$-maximal regularity (respectively, is $B^{s,q}_{L^p}$-well-posed) for some $p\in (1,\infty)$ and $q\in [1,\infty]$.

\item [(iii)] For every $k\in \ZZ$ the operator $b(k):=-k^2P + ik B + A + \hat c(k)$ is bijective and the sequences 
\[
\left( b(k)^{-1}\right)_{k\in \ZZ}\subset \cL(X,Z), \,\, (k B b(k)^{-1})_{k\in \ZZ}\subset \cL(X), \,\, (k^2P b(k)^{-1})_{k\in \ZZ}\subset \cL(X) 
\]
are bounded (respectively,  in addition, $\left(kb(k)^{-1}\right)_{k\in \ZZ}\subset \cL(X)$ is bounded).

\end{itemize}

\emph{(b)} Assume that $X$ has the $U\!M\!D$ property and  $\hat c\in \m^1_\cR(\ZZ;\cL(Z,X))$. Then, the following assertions are equivalent.

\begin{itemize}

\item [(i)] The problem \eqref{aee} has $\E$-maximal regularity (respectively, is $\E$-well-posed) for every $\E$ such that  
\[
\E \in \bigl\{ \Phi,\, B^{s,q}_\Psi,\, F^{s,r}_\Phi: s\in \R,  q\in [1,\infty], r\in (1,\infty),\,\, \Phi, \Phi', \Psi \in \cL_{M} \bigr\}.
\]
(respectively, in addition, $\E\subset L^1$).

\item [(ii)] The problem \eqref{aee} has $L^p$-maximal regularity (respectively, is $\E$-well-posed) for some $p\in (1,\infty)$.
 
\item [(iii)] For every $k\in \ZZ$ the operator $b(k):=-k^2P + ik B + A + \hat c(k)$ is bijective and the sequences 
\[
(b(k)^{-1})_{k\in \ZZ}\subset \cL(X,Z), \quad (k B b(k)^{-1})_{k\in \ZZ}\subset \cL(X), \quad (k^2P b(k)^{-1})_{k\in \ZZ}\subset \cL(X) 
\]
are $\cR$-bounded (respectively, in addition,  $\left(kb(k)^{-1}\right)_{k\in \ZZ}\subset \cL(X)$ is $\cR$-bounded). 
\end{itemize}
\end{theorem}

\section{Particular forms of \eqref{aee}}\label{last}
In this section we specialize our general results from the previous section to particular forms of the abstract problem \eqref{aee}, which have been studied in the literature.
In particular, Theorems \ref{charact of mr} and \ref{charact of wp} extend many results from  a long series of articles, where such characterisations have been studied progressively; \cite{DaPrLu86, Pr93, ArBu02, ArBu04, KeLi04, Li06, BuFa08, LiPo08, Po09, BuFa09, KeLiPo09, HeLi12, LiPo13, FuLi14, BuCa17, BuCa18, BuCa19} and the references therein.

\subsection{Integro-differential equations}\label{with c}

 Here, we consider \eqref{aee} for special classes of convolutors $c$, which arise in the (abstract) reformulation of the integro-differential equations describing physical processes in materials with fading memory.
We start with equations, where a convolutor $c$ is given by the so-called {\it finite delay operators}, that is, $c$ is of the form   
\begin{equation}\label{delay op}
c\ast u := Hu_{\cdot} + Gu'_{\cdot}, 
\end{equation} where $H, G \in \cL( L^p(\T;X), X)$ for some $p\in (1,\infty)$, $u_{t}(s):=u(t+s)$ and, if, for instance, $u\in W^{1,p}(\T;X)$,  $u'_{t}(s) = u'(t+s)$ ($s\in [0,2\pi]$).   
Note that if for all $k\in \ZZ$ and $x\in X$ we set
\[
H_kx:=H(e_k\otimes x) \quad \textrm{ and }\quad G_kx := G(e_k\otimes x),
\]
then $\widehat{Hu_{\cdot}}(k) = H_k\hat{u}(k)$ and $\widehat{Gu_{\cdot}}(k) = G_k\hat{u}(k)$ for all $u\in L^p(X)$. Consequently, the sequences 
$h:=(H_k)_{k\in \ZZ}$ and $g:=(G_k)_{k\in \ZZ}$ are in $\cM_p(\ZZ; X,X)$. It shows that one can extend the meaning of \eqref{delay op} to all $u\in \cD'(X)$ putting 
\[
c\ast u = h(\Delta) u + \partial g(\Delta) u = h(\Delta) u +  g(\Delta)(\partial u),   
\] that is,  $\hat{c}(k) = h(k) + ikg(k)$, $k\in \ZZ$. 

This abstract reformulation of \eqref{delay op} leads to the following characterisation of the maximal regularity and well-posedness for this class of equations.

\begin{corollary}\label{mr thm delay op}
Let $A$, $B$ and $P$ be closed, linear operators on a Banach space $X$. 
Let $(h(k))_{k\in \ZZ}$, $(g(k))_{k\in \ZZ}\subset \cL(X)$. 

\emph{(1)} Assume that $(\cdot)^{-\alpha}h\in \m^{2}(\ZZ;\cL(X))$ and $(\cdot)^{1-\alpha}g\in \m^{2}(\ZZ;\cL(X))$ for some $\alpha\in \{0,1,2\}$. Let $\hat c := h + i(\cdot) g$. 
Then, the following statements are true.
\begin{itemize}
\item [(i)] If $\alpha = 0$, then the equivalence $(i)\Leftrightarrow(ii)\Leftrightarrow (iii)$ of Theorem \ref{charact of mr}(a) holds for \eqref{aee} with $\hat c$.
\item [(ii)] If $\alpha = 1$, then the conclusion of Theorem \ref{charact of wp}(a) holds, and if, in addition, $B=I$, then the conclusion of Theorem \ref{charact of mr}(a) holds too. 
\item [(iii)] If $\alpha = 2$ and, in addition, $P=I$, then the conclusion of Theorem \ref{charact of mr}(a) and Theorem \ref{charact of wp}(a) hold.
\end{itemize}

\emph{(2)} Let $X$ have the $U\!M\!D$ property. Assume that $(\cdot)^{-\alpha}h\in \m_\cR^{1}(\ZZ;\cL(X))$ and $(\cdot)^{1-\alpha}g\in \m^{1}_\cR(\ZZ;\cL(X))$ for some $\alpha\in \{0,1,2\}$. Then the analogous statements to those of the point (1) hold, that is, the statements which are obtained by replacing the assertion (a) of Theorems \ref{charact of mr} and \ref{charact of wp} with the corresponding (b) therein.
\end{corollary}

\begin{remark}
(a) One can readily seen, that Corollary \ref{mr thm delay op} extends several related results from the literature, where the convolutors $c$ are given by finite delay operators; see, e.g. \cite{BuCa19}, \cite{BuCa18} and also references therein. 

(b) Moreover, Theorem \ref{mr thm delay op}(2) covers \cite[Theorem 3.4]{FuLi14}. Indeed, recall that Fu and Li in \cite{FuLi14} studied the problem \eqref{aee} with $P=I$ and the convolutor $c$ given by the {\it infinite} delay operators, that is, 
\[
c\ast u = H u_\cdot + G u'_\cdot,
\] where $H, G:\mathscr{B}  \rightarrow X$ are bounded, linear operators, $\mathscr{B}$ is a space of $X$-valued functions on $\R_-:=(-\infty, 0]$, which is axiomatically defined in \cite{FuLi14}, and $u_\cdot$ is given as before by $u_t(s):= u(t+s)$ $(s\in \R_-)$. Since $e_k$, $k\in \ZZ$, considered as functions on $\R_-$, belong to $\cC_b(\R_-)\subset \mathscr{B}$, one can define sequences $h$ and $g$ similarly as above, i.e. $h(k) x := H(e_k(\cdot)x)$ and $g(k)x :=(e_k(\cdot)x)$, $k\in \ZZ$. Then, $c \in \cD'(X)$ with $\hat c(k) = h(k) + ik g(k)$, $k\in \ZZ$.  
The axioms of $\mathscr B$ easily give that $h$ and $g$ are $\cR$-bounded in $\cL(X)$; see e.g. the proof of \cite[Lemma 3.2]{FuLi14}. 
Finally, one can easily check that the condition $(ii)$ of \cite[Theorem 3.4]{FuLi14} implies the assumptions of Corollary \ref{mr thm delay op}(2) for $\alpha = 2$ (with $P=I$). 

Analogously, one can check that Theorem \ref{mr thm delay op}(a) extends \cite[Theorems 4.4~and~4.7]{FuLi14}; see also references therein, as well corresponding results in \cite{KeLi05, BuFa08, BuFa09, LiPo08, HeLi12}.
\end{remark}

To illustrate Theorem \ref{charact with Z}, we remark on the integro-differential equations with infinite delays, which have been studied in \cite{KeLi04}. Namely, in the problem studied in \cite{KeLi04} (i.e., \eqref{aee} with $P=0$ and $B=I$) the infinite delay operator is given by 
\[
c\ast u: = d(\cdot) A \ast u, 
\]
where $d\in L^1_{loc}(\R_+)\subset L^1_{loc}(\R)$ is such that the Fourier transform of $d$ exists at points $k\in \ZZ$ and the sequence $(\cF d(k))_{k\in \ZZ}$ satisfies some further assumptions corresponding to a {\it $\gamma$-regularity} ($\gamma = 1,2)$; see \cite[Theorems 2.12 and 3.9]{KeLi04}). It is readily seen that, these assumptions imply $\hat c \in \m^\gamma_\cR (\ZZ; \cL(Z,X))$ with $Z:=D_A$ ($\gamma = 1,2$). Therefore, Theorem \ref{charact with Z} might also be seen as an extension of those results.

\subsection{Differential equations}\label{diff equ}  

We conclude with some remarks on \eqref{aee} with $c \equiv 0$, that is, 
\begin{equation}\label{aee no c}
\partial \cP \partial u + \cB \partial u + \cA u = f \qquad (\textrm{in } \cD'(X)).
\end{equation} 
The particular form of \eqref{aee no c} with $\cP = 0$ and $\cB = I$  was, in a sense, a prototype for further periodic extensions done in the literature; see Arendt and Bu \cite[Theorem 2.3 and Corollary 2.4]{ArBu02} for the corresponding result on the $L^p$-well-posedness and its $p$-{\it{independence}}. For the studies of the other forms of \eqref{aee no c} in connection to their well-posedness in the context of the classical Lebesgue-Bochner, Besov, and Triebel-Lizorkin spaces (i.e., corresponding to $\Phi=L^p$) see for instance \cite[Section 2]{ArBu02}, \cite{ArBu04, LiPo11, Bu13}, as well as the references provided for the integro-differential equations.  

Theorem \ref{charact of wp} (for $c\equiv 0$), in particular, gives an extension of those results and also provides periodic variant of extrapolation results known in the euclidean setting; see, e.g. \cite{PrSi07, AuAx11, ChFi14, ChKr18, FaHyLi20, Kr22} and references therein. We  leave the formulation of this result to the interested reader. 
Finally, in articles mentioned in this section the reader can find several concrete models of equations arising in the applied mathematics. Such equations with periodic boundary conditions have been studied therein in the context of the well-posedness with respect to the classical Lebesgue-Bochner $L^p$, Besov $B^{s,q}_p$, and Triebel-Lizorkin $F^{s,q}_p$ spaces. The same equations could serve as an illustration of our abstract results, Theorems \ref{charact of mr}, \ref{charact of wp}, and \ref{charact with Z}.\\
 
\noindent\textbf{Declaration of interest:} none.\\

\noindent\textbf{Data availability:} no data was used for the research described in the article.

\providecommand{\href}[2]{#2}


\begin{thebibliography}{10}

\bibitem{Am95}
H.~Amann, \emph{Linear and {Q}uasilinear {P}arabolic {P}roblems: {A}bstract
  {L}inear {T}heory}, Monographs in Mathematics, vol.~89, Birkh\"auser,
  Basel, 1995.

\bibitem{Am97}
H.~Amann, \emph{Operator-valued {F}ourier multipliers, vector-valued
  {B}esov spaces, and applications}, Math. Nachr. \textbf{186} (1997), 5-56.

\bibitem{AnMo09}
K.~F.~Andersen and P.~Mohanty, \emph{Restriction and extension of
  {F}ourier multipliers between weighted {$L^p$} spaces on {$\Bbb R^n$} and
  {$\Bbb T^n$}}, Proc. Amer. Math. Soc. \textbf{137} (2009), 1689-1697.

\bibitem{KeAp20}
R.~Aparicio and V.~Keyantuo, \emph{Besov maximal regularity for a
  class of degenerate integro-differential equations with infinite delay in
  {B}anach spaces}, Math. Methods Appl. Sci. \textbf{43} (2020),
  7239-7268. 

\bibitem{ApKe20}
R.~Aparicio and V.~Keyantuo, \emph{{$L^p$}-maximal regularity for a class of degenerate
  integro-differential equations with infinite delay in {B}anach spaces}, J.
  Fourier Anal. Appl. \textbf{26} (2020), pp. 39. 

\bibitem{ArBaBu04}
W.~Arendt, C.~ Batty, and S.~ Bu, \emph{Fourier multipliers for
  {H}\"{o}lder continuous functions and maximal regularity}, Studia Math.
  \textbf{160} (2004), 23-51. 

\bibitem{ArBu02}
W.~Arendt and S.~Bu, \emph{The operator-valued {M}arcinkiewicz
  multiplier theorem and maximal regularity}, Math. Z. \textbf{240} (2002), 311-343. 

\bibitem{ArBu04}
W.~Arendt and S.~Bu,  \emph{Operator-valued {F}ourier multipliers on periodic {B}esov spaces
  and applications}, Proc. Edinb. Math. Soc. (2) \textbf{47} (2004),
  15-33. 

\bibitem{ArBu04a}
W.~Arendt and S.~Bu,  \emph{Operator-valued multiplier theorems characterizing {H}ilbert
  spaces}, J. Aust. Math. Soc. \textbf{77} (2004), 175-184.

\bibitem{AsBeGi94}
N.~Asmar, E.~Berkson, and T.~A.~Gillespie, \emph{On {J}odeit's
  multiplier extension theorems}, J. Anal. Math. \textbf{64} (1994), 337-345.

\bibitem{AuAx11a}
P.~Auscher and A.~Axelsson, \emph{Remarks on maximal regularity},
  Parabolic problems, Progr. Nonlinear Differential Equations Appl., vol.~80,
  Birkh\"auser/Springer Basel AG, Basel, 2011, pp.~45-55. 

\bibitem{AuAx11}
P.~Auscher and A.~Axelsson,  \emph{Weighted maximal regularity estimates and solvability of non-smooth elliptic systems {I}}, Invent. Math. \textbf{184} (2011), 47-115. 

\bibitem{BeSh88}
C.~Bennett and R.~Sharpley, \emph{Interpolation of operators}, Pure and
  Applied Mathematics, vol. 129, Academic Press, Inc., Boston, MA, 1988.

\bibitem{BeGi03}
E.~Berkson and T.~A.~Gillespie, \emph{On restrictions of multipliers in
  weighted settings}, Indiana Univ. Math. J. \textbf{52} (2003),
  927-961.

\bibitem{BLVi13}
O.~Blasco and P.~Villarroya, \emph{Transference of vector-valued
  multipliers on weighted {$L^p$}-spaces}, Canad. J. Math. \textbf{65} (2013), 510-543. 

\bibitem{BoKa97}
A.~B\"{o}ttcher and Y.~I.~Karlovich, \emph{Carleson curves,
  {M}uckenhoupt weights, and {T}oeplitz operators}, Progress in Mathematics,
  vol. 154, Birkh\"{a}user, Basel, 1997. 

\bibitem{Bo82}
J.~Bourgain, \emph{A {H}ausdorff-{Y}oung inequality for {B}-convex {B}anach
  spaces}, Pacific J. Math. \textbf{101} (1982), 255-262.

\bibitem{Bo83}
J.~Bourgain, \emph{Some remarks on {B}anach spaces in which martingale difference
  sequences are unconditional}, Ark. Math. \textbf{21} (1983), 163-168.

\bibitem{Bu13}
S.~Bu, \emph{Well-posedness of second order degenerate differential
  equations in vector-valued function spaces}, Studia Math. \textbf{214}
  (2013), 1-16. 

\bibitem{BuCa17}
S.~Bu and G.~Cai, \emph{Well-posedness of second-order degenerate
  differential equations with finite delay in vector-valued function spaces},
  Pacific J. Math. \textbf{288} (2017), 27-46. 

\bibitem{BuCa18}
S.~Bu and G.~Cai, \emph{Periodic solutions of second order degenerate differential
  equations with delay in {B}anach spaces}, Canad. Math. Bull. \textbf{61}
  (2018), 717-737. 

\bibitem{BuCa19}
S.~Bu and G.~Cai, \emph{Periodic solutions of second order degenerate differential
  equations with finite delay in {B}anach spaces}, J. Fourier Anal. Appl.
  \textbf{25} (2019), 32-50. 

\bibitem{BuFa08}
S.~Bu and Y.~Fang, \emph{Maximal regularity for integro-differential
  equation on periodic {T}riebel-{L}izorkin spaces}, Taiwanese J. Math.
  \textbf{12} (2008), 281-292. 

\bibitem{BuFa09}
S.~Bu and Y.~Fang, \emph{Periodic solutions of delay equations in {B}esov spaces and
  {T}riebel-{L}izorkin spaces}, Taiwanese J. Math. \textbf{13} (2009),
  1063-1076. 

\bibitem{BuKi04}
S.~Bu and J.~M.~ Kim, \emph{Operator-valued {F}ourier multiplier
  theorems on {$L_p$}-spaces on {$\Bbb T^d$}}, Arch. Math. (Basel) \textbf{82}
  (2004), 404-414.

\bibitem{Bu81}
A.~V.~ Bukhvalov, \emph{Hardy spaces of vector-valued functions}, J. Soviet
  Math. \textbf{16} (1981), 1051-1059.

\bibitem{ChFi14}
R.~Chill and A.~Fiorenza, \emph{Singular integral operators with
  operator-valued kernels, and extrapolation of maximal regularity into
  rearrangement invariant {B}anach function spaces}, J. Evol. Eq. (2014), 795-828.

\bibitem{ChSr05}
R.~Chill and S.~Srivastava, \emph{{$L^p$}-maximal regularity for second order
  {C}auchy problems}, Math. Z. \textbf{251} (2005), 751-781.

\bibitem{ChKr18}
R.~Chill and S.~Kr\'{o}l, \emph{Weighted inequalities for singular
  integral operators on the half-line}, Studia Math. \textbf{243} (2018), 171-206. 

\bibitem{CdPSW00}
P.~Cl\'{e}ment, B.~de~Pagter, F.~A.~ Sukochev, and H.~Witvliet, \emph{Schauder
  decomposition and multiplier theorems}, Studia Math. \textbf{138} (2000), 135-163. 

\bibitem{ClPr01}
P.~Cl\'{e}ment and J.~Pr\"{u}ss, \emph{An operator-valued transference
  principle and maximal regularity on vector-valued {$L_p$}-spaces}, Evolution
  equations and their applications in physical and life sciences ({B}ad
  {H}errenalb, 1998), Lecture Notes in Pure and Appl. Math., vol. 215, Dekker,
  New York, 2001, pp.~67-87. 

\bibitem{DaPrLu86}
G.~Da~Prato and A.~ Lunardi, \emph{Periodic solutions for linear
  integro-differential equations with infinite delay in {B}anach spaces},
  Differential equations in {B}anach spaces ({B}ologna, 1985), Lecture Notes in
  Math., vol. 1223, Springer, Berlin, 1986, pp.~49-60. 

\bibitem{deLe65}
K.~de~Leeuw, \emph{On {$L_{p}$} multipliers}, Ann. of Math. (2) \textbf{81}
  (1965), 364-379. 

\bibitem{DeHiPr03a}
R.~Denk, M.~Hieber, and J.~Pr\"uss, \emph{Optimal {$L_p-L_q$}-regularity for
  vector-valued parabolic problems with inhomogeneous boundary data},
  Submitted, 2003.

\bibitem{Ed67}
R.~E. Edwards, \emph{Fourier series: a modern introduction. {V}ol. {II}}, Holt,
  Rinehart and Winston, Inc., New York-Montreal, Que.-London, 1967.

\bibitem{Ed79}
R.~E. Edwards, \emph{Fourier series. {A} modern introduction. {V}ol. 1}, second ed.,
  Graduate Texts in Mathematics, vol.~64, Springer, New York-Berlin,
  1979. 

\bibitem{FaHyLi20}
S.~Fackler, T.~Hyt\"{o}nen, and N.~Lindemulder, \emph{Weighted
  estimates for operator-valued {F}ourier multipliers}, Collect. Math.
  \textbf{71} (2020), 511-548. 

\bibitem{FuLi14}
X.~Fu and M.~Li, \emph{Maximal regularity of second-order evolution
  equations with infinite delay in {B}anach spaces}, Studia Math. \textbf{224}
  (2014), 199-219. 

\bibitem{GiWe03}
M.~Girardi and L.~Weis, \emph{Operator-valued {F}ourier multiplier
  theorems on {B}esov spaces}, Math. Nachr. \textbf{251} (2003), 34-51.

\bibitem{HeLi12}
H.~R.~Henr\'{\i}quez and C.~Lizama, \emph{Periodic solutions of
  abstract functional differential equations with infinite delay}, Nonlinear
  Anal. \textbf{75} (2012), 2016-2023.

\bibitem{HNVW16}
T.~Hyt\"{o}nen, J.~van~Neerven, M.~Veraar, and L.~Weis, \emph{Analysis
  in {B}anach spaces. {V}ol. {I}. {M}artingales and {L}ittlewood-{P}aley
  theory}, Ergebnisse der Mathematik und ihrer Grenzgebiete. 3. Folge. A Series
  of Modern Surveys in Mathematics [Results in Mathematics and Related Areas.
  3rd Series. A Series of Modern Surveys in Mathematics], vol.~63, Springer,
  Cham, 2016. 

\bibitem{HyNeVeWa17}
T.~Hyt\"{o}nen, J.~van~Neerven, M.~Veraar, and L.~Weis, \emph{Analysis in {B}anach spaces. {V}ol. {II}. {P}robabilistic methods and operator theory}, Ergebnisse der
  Mathematik und ihrer Grenzgebiete. 3. Folge. A Series of Modern Surveys in
  Mathematics [Results in Mathematics and Related Areas. 3rd Series. A Series
  of Modern Surveys in Mathematics], vol.~67, Springer, Cham, 2017,

\bibitem{HyWe06}
T.~Hyt\"{o}nen and L.~Weis, \emph{Singular integrals on {B}esov spaces},
  Math. Nachr. \textbf{279} (2006), 581-598.

\bibitem{HyWe07}
T.~Hyt\"{o}nen and L.~Weis,  \emph{Singular convolution integrals with operator-valued kernel},
  Math. Z. \textbf{255} (2007), 393-425. 

\bibitem{Hy12}
T.~Hyt\"{o}nen, \emph{The sharp weighted bound for general
  {C}alder\'{o}n-{Z}ygmund operators}, Ann. of Math. (2) \textbf{175} (2012), 1473-1506. 

\bibitem{Jo70}
M.~Jodeit, Jr., \emph{Restrictions and extensions of {F}ourier multipliers},
  Studia Math. \textbf{34} (1970), 215-226.

\bibitem{KeLi04}
V.~Keyantuo and C.~Lizama, \emph{Fourier multipliers and
  integro-differential equations in {B}anach spaces}, J. London Math. Soc. 
  \textbf{69} (2004), 737-750. 

\bibitem{KeLi05}
V.~Keyantuo and C.~Lizama, \emph{Maximal regularity for a class of integro-differential equations
  with infinite delay in {B}anach spaces}, Studia Math. \textbf{168} (2005), 25-50. 

\bibitem{KeLiPo09}
V.~Keyantuo, C.~Lizama, and V.~Poblete, \emph{Periodic
  solutions of integro-differential equations in vector-valued function
  spaces}, J. Differential Equations \textbf{246} (2009), 1007-1037.

\bibitem{KnMcMo16}
G.~Knese, J.~E.~McCarthy, and K.~Moen, \emph{Unions of {L}ebesgue spaces
  and {$A_1$} majorants}, Pacific J. Math. \textbf{280} (2016),
  411-432. 

\bibitem{Kr22}
S.~Kr\'ol, \emph{The maximal regularity property of abstract
  integro-differential equations}, arXiv:2209.06630

\bibitem{KuWe04}
P.~C.~Kunstmann and L.~Weis, \emph{Maximal {$L^p$} regularity for parabolic
  equations, {F}ourier multiplier theorems and {$H^\infty$} functional
  calculus}, Levico Lectures, Proceedings of the Autumn School on Evolution
  Equations and Semigroups (M. Iannelli, R. Nagel, S. Piazzera eds.), vol.~69,
  Springer Verlag, Heidelberg, Berlin, 2004, pp.~65-320.

\bibitem{Li06}
C.~Lizama, \emph{Fourier multipliers and periodic solutions of delay
  equations in {B}anach spaces}, J. Math. Anal. Appl. \textbf{324} (2006), 921-933. 
  
\bibitem{LiPo08}
C.~Lizama and V.~Poblete, \emph{Maximal regularity for perturbed
  integral equations on periodic {L}ebesgue spaces}, J. Math. Anal. Appl.
  \textbf{348} (2008), 775-786. 

\bibitem{LiPo11}
C.~Lizama and R.~Ponce, \emph{Periodic solutions of degenerate
  differential equations in vector-valued function spaces}, Studia Math.
  \textbf{202} (2011), 49-63. 

\bibitem{LiPo13}
C.~Lizama and R.~Ponce,
 \emph{Maximal regularity for degenerate differential equations with
  infinite delay in periodic vector-valued function spaces}, Proc. Edinb. Math.
  Soc. (2) \textbf{56} (2013), 853-871. 

\bibitem{Po09}
V.~Poblete, \emph{Maximal regularity of second-order equations with
  delay}, J. Differential Equations \textbf{246} (2009), 261-276.

\bibitem{Pr93}
J.~Pr\"uss, \emph{Evolutionary {I}ntegral {E}quations and {A}pplications},
  Monographs in Mathematics, vol.~87, Birkh\"auser Verlag, Basel, 1993.

\bibitem{PrSi04}
J.~Pr\"uss and G.~Simonett, \emph{Maximal regularity for evolution equations in
  weighted {$L\sb p$}-spaces}, Arch. Math. (Basel) \textbf{82} (2004),
  415-431.

\bibitem{PrSi07}
J.~Pr\"uss and G.~Simonett,  \emph{{$H^\infty$} calculus for the sum of non-commuting operators},
  Trans. Amer. Math. Soc. \textbf{359} (2007), 3549-3565.

\bibitem{RdeF86}
J.~L.~Rubio~de Francia, \emph{Martingale and integral transforms of
  {B}anach space valued functions}, Probability and {B}anach spaces
  ({Z}aragoza, 1985), Lecture Notes in Math., vol. 1221, Springer, Berlin,
  1986, pp.~195-222.

\bibitem{RuRuTo86}
J.~L.~Rubio~de Francia, Francisco~J. Ruiz, and Jos{\'e}~L. Torrea,
  \emph{Calder\'on-{Z}ygmund theory for operator-valued kernels}, Adv. in Math.
  \textbf{62} (1986), 7-48. 

\bibitem{ScTr87}
H.~J.~Schmeisser, H.~Triebel, \emph{Topics in {F}ourier analysis
  and function spaces}, A Wiley-Interscience Publication, John Wiley \& Sons,
  Ltd., Chichester, 1987. 

\bibitem{Si64}
L. de~Simon, \emph{Un’applicazione della teoria degli integrali singolari allo studio delle equazioni differenziali lineari astratte del primo ordine}, Rend. Sem. Mat. Univ. Padova \textbf{34} (1964), 205-223.

\bibitem{So64}
P.~E.~Sobolevskii, \emph{Coerciveness inequalities for abstract parabolic
  equations}, Dokl. Akad. Nauk SSSR \textbf{157} (1964), 52-55.

\bibitem{St93}
E.~M.~Stein, \emph{Harmonic analysis: real-variable methods, orthogonality,
  and oscillatory integrals}, Princeton Mathematical Series, vol.~43, Princeton
  University Press, Princeton, NJ, 1993, With the assistance of Timothy S.
  Murphy, Monographs in Harmonic Analysis, III. 

\bibitem{StWe07}
{\v{Z}}.~{\v{S}}trkalj and L.~Weis, \emph{On operator-valued {F}ourier
  multiplier theorems}, Trans. Amer. Math. Soc. \textbf{359} (2007),
  3529-3547. 


\bibitem{Triebel78}
H.Triebel, \emph{Spaces of Besov-Hardy-Sobolev Type}, Teubner-Texte Math. {\bf 15}, Leipzig: Teubner 1978.



\bibitem{We01}
L.~Weis, \emph{Operator-valued {F}ourier multiplier theorems and maximal
  {$L_p$}-regularity}, Math. Ann. \textbf{319} (2001), 735-758.

\end{thebibliography}
\end{document}